\documentclass[a4paper]{article}

\usepackage[margin=0.5in]{geometry}
\usepackage{microtype}
\usepackage[utf8]{inputenc}
\usepackage[english]{babel}
\usepackage{cite}
\usepackage[nottoc]{tocbibind}
\usepackage{amsmath}
\usepackage{amsfonts}
\usepackage{amssymb}
\usepackage{amsthm}
\usepackage{anysize}
\usepackage{centernot}
\usepackage{graphicx}
\usepackage{dcolumn}
\usepackage{enumitem}
\usepackage{tikz}
\usepackage{tikz-cd}
\usepackage{ytableau}
\usepackage{mathtools}
\usepackage{genyoungtabtikz}
\newtheorem{theorem}{Theorem}[section]
\newtheorem*{theorem*}{Theorem}
\newtheorem{lemma}[theorem]{Lemma}
\newtheorem{prop}[theorem]{Proposition}
\newtheorem{cor}[theorem]{Corollary}
\newtheorem{property}[theorem]{Property}
\theoremstyle{definition}
\newtheorem{defin}[theorem]{Definition}
\newtheorem{oss}[theorem]{Observation}
\theoremstyle{example}

\DeclareMathOperator{\diag}{Diag}

\DeclareMathOperator{\id}{Id}
\DeclareMathOperator{\R}{R}
\DeclareMathOperator{\re}{r}
\DeclareMathOperator{\pl}{PL}
\DeclareMathOperator{\mi}{MIN}

\DeclareMathOperator{\adj}{Ad}
\DeclareMathOperator{\N}{N}
\newcommand{\cdr}[1]{\overline{#1}}

\DeclareMathOperator{\stab}{Stab}

\newcommand{\matr}[1]{\mathbf{#1}}

\DeclareMathOperator{\RM}{RM}

\newtheorem{manualtheoreminner}{Theorem}
\newenvironment{manualtheorem}[1]{%
  \renewcommand\themanualtheoreminner{#1}%
  \manualtheoreminner
}{\endmanualtheoreminner}

\title{On the Bruhat $\mathcal{G}$-order between local systems on the $B$-orbits of a Hermitian symmetric variety}
\date{August 2020}
\author{Michele Carmassi}
\begin{document}
\maketitle
\section{Introduction}
Let $G$ be a semisimple linear algebraic group over $\mathbb{C}$.
Fix a maximal torus $T\subseteq G$ and suppose that the root system $\Phi=\Phi(G,T)$ is irreducible.
Then, fix a basis $\Delta$ for $\Phi$ or equivalently a Borel subgroup $T\subseteq B\subseteq G$.
This gives a partition  of $\Phi=\Phi^+\sqcup -\Phi^+$ where $\Phi^+$ are the positive roots.
Let $P\supseteq B$ be a parabolic subgroup with Levi decomposition $P=L\ltimes P^u$ where $L$ is the Levi subgroup and $P^u$ is the unipotent radical of $P$.
Suppose $P^u$ abelian (this gives restrictions on the root system and the parabolic $P$).
Then $L$ is the set of fixed points for an involution of $G$ and the quotient $G/L$ is said to be an \textit{Hermitian symmetric varieties}.
The Borel $B$ acts by multiplication on $G/L$ and this action has finitely many orbits.
These orbits are the first ingredient of our study.

The situation is quite similar to the action of $B$ on the flag variety $G/B$, which was studied by Chevalley.
We know that every roots $\alpha\in\Phi$ determines a reflection $s_\alpha$ which is the linear map on $\Phi\otimes \mathbb{R}$ that fixes the hyperplane orthogonal to $\alpha$ and sends $\alpha$ to $-\alpha$.
The group generated by these reflections is called Weyl group and is denoted with $W$.
It is well known that the orbits in $G/B$ are parametrized by the elements of $W$.
Moreover, there is an isomorphism between $W$ and the group $\N_G(T)/T$ where $\N_G(T)$ is the normalizer of $T$ in $G$.
With this in mind, the orbit corresponding to $\omega\in W$ is exactly $BwB/B$ where $w$ is any representative in $\N_G(T)$ of $\omega$.
More in general, there exists a similar parametrization for the $B$-orbits on the partial flag variety $G/P$ for $P\supseteq B$ a parabolic subgroup.
Recall that the parabolic subgroups $P\supseteq G$ correspond to the subsets $S\subseteq \Delta$.
From this, we can associate two subsets of $W$ to the parabolic subgroup.
The first one is $W_P$, which is the subgroup generated by the reflections $s_\alpha$ with $\alpha\in S$, and the second one is $W^P$ which is the set of minimal length representatives for the cosets of $W_P$ in $W$.
With this notation and the correspondence above in mind we then have the decomposition 
\[G/P=\bigcup_{w\in W^P}BwP/P\]

The inclusion $L\subseteq P$ gives a map $G/L\longrightarrow G/P$ that is $B$-equivariant.
It follows that the $B$-orbits in $G/L$ can be parametrized by the elements of $W^P$ plus some other data.
The parametrization of the $B$-orbits on a Hermitian symmetric variety is due to Richardson and Springer (Theorem 5.2.4, \cite{RS2}), but we will use the description by Gandini and Maffei (\cite{GM}) which is a bit different.
For this parametrization, the additional data is a set $S\subseteq \Phi^+$ of positive roots which are mutually orthogonal and such that $v(S)<0$.
These roots all live in the subset $\Psi\doteqdot\left\lbrace \beta\in \Phi^+\mid \exists v\in W^P v(\beta)<0\right\rbrace$.
The pairs $(v,S)$ obtained this way are called \textit{admissible pairs}.

For every $\alpha\in\Psi$ fix a generator $e_\alpha$ for the root space $\mathfrak{u}_\alpha$ and if $S\subseteq \Psi$ write $e_S$ for $\sum_{\alpha\in S}e_\alpha$.
Finally, put $x_S=\exp(e_S)L/L$.
Then the correspondence in \cite{GM} is
\begin{align*}
\left\lbrace (v,S) \text{ admissible}\right\rbrace &\leftrightarrow \left\lbrace B-\text{orbits in }G/L\right\rbrace\\
(v,S) &\mapsto Bvx_S
\end{align*}
More details on this parametrization and its properties can be found in section \ref{orbits}.

In $G/B$, as well as in $G/L$, we may order the orbits with respect to the inclusion of the closures.
That is: $\mathcal{O}<\mathcal{O}'$ if and only if $\mathcal{O}\subseteq \cdr{\mathcal{O}'}$.
We obtain the so-called \textit{Bruhat order} on the $B$-orbits.
In the flag variety case, this induces an order on $W$ which is still called Bruhat order and that has a well known combinatorial characterization in terms of sub-expression of $w\in W$ when $w$ is written as a product of the simple reflections (the reflections associated to the roots in the basis) in a reduced way.
The characterization of the Bruhat order for $G/L$ is Theorem \ref{ordineGM}.
It was originally conjectured by Richardson and Springer and has been proved by Gandini and Maffei in \cite{GM}.

The second ingredient of this study are the $B$-equivariant $\mathbb{C}$-local systems of rank $1$ over the $B$-orbits.
A $\mathbb{C}$-local system of rank $1$ is a sheaf which is locally isomorphic to the constant sheaf $\mathbb{C}$.
Another way to see a local system is as a complex line bundle $\pi\colon E\longrightarrow\mathcal{O}$ with flat connection which locally trivializes as $E|_U\cong U\times \mathbb{C}$ where on $\mathbb{C}$ we consider the discrete topology.
Note that the existence of a flat connection is equivalent to $(E,\pi)$ having a trivializing open covering $\left\lbrace U_i\right\rbrace_i$ of $\mathcal{O}$ such that the transition functions are constant.
We also ask that $E$ admits a $B$-action that commutes with the $B$-action on $\mathcal{O}$, namely, that $E$ is \textit{$B$-equivariant}.
Finally, we are not interested in every local system per se, but only in their isomorphism classes.
The set of pairs $(\mathcal{O},\gamma)$ where $\mathcal{O}$ is a $B$-orbit on $G/L$ and $\gamma$ is an isomorphism class of $B$-equivariant local systems over $\mathcal{O}$ will be denoted with $\mathcal{D}$.
Following Lusztig and Vogan (\cite{Vogan} and \cite{LV}) we may put an order in $\mathcal{D}$ which is called \textit{Bruhat G-order} (definition \ref{ordinelb}).

The aim of this paper is to study the Bruhat $\mathcal{G}$-order.
In particular we want to find a combinatorial characterization of this order and we want to study the associated Hasse diagram.

The first general result is that if we define the subset $\mathcal{D}_0=\left\lbrace \left(\mathcal{O},\gamma\right)\mid \gamma \text{ is trivial}\right\rbrace$, then the Bruhat $\mathcal{G}$-order restricted to $\mathcal{D}_0$ coincides with the Bruhat order on the orbits (Proposition \ref{ordtriv}).
We then need to study the Bruhat $\mathcal{G}$-order when there are orbits that admit non-trivial root systems.
The results will depend on the type of the root system $\Phi$, but also on the group $G$.
More precisely, with some calculation we can see that if $G$ is adjoint, then all the local systems are trivial.
We then focus our attention on $G$ simply connected.
We will see that in this case we can find non-trivial local systems even though we may need some additional hypothesis.

If $\Phi$ is simply laced we will show that the orbits admit non-trivial root systems if and only if $\Psi$ verifies an additional property (Property \ref{unic}).
In this case we say that an orbit $(v,S)$ is of \textit{maximum rank} if $S$ is maximal among the orthogonal subsets of $\Psi$.
The following result characterizes completely the Bruhat $\mathcal{G}$-order in the simply laced case and it is probably the most interesting result in this paper.
\begin{manualtheorem}{\ref{risADE}}
Suppose that the linear algebraic group $G$ is simply connected and the root system $\Phi$ is simply laced.
If $\Psi$ doesn't verify Property \ref{unic}, then all local systems are trivial and $\mathcal{D}=\mathcal{D}_0$.

If instead $\Psi$ verifies Property \ref{unic}, then:
\begin{enumerate}
\item the orbits of maximum rank admit exactly two non-isomorphic local systems, one being trivial and one being non-trivial.
The other orbits admit only the trivial local system;
\item the subset of all the orbits with trivial local system is a connected component of the Hasse diagram, while the subset of the orbits of maximum rank with non-trivial local system is another connected component;
\item in every connected component, the Bruhat $\mathcal{G}$-order between the elements coincides with the Bruhat order between the underlying orbits.
\end{enumerate}
\end{manualtheorem}

If $\Phi$ is of type $\matr{B}$ the situation is similar to the simply laced cases.
The only orbits that admit non-trivial local systems are the orbits of maximum rank and again all the non-trivial local systems are isomorphic.
If $(v,S)$ is an admissible pair and hence an orbit, denote with $H(v,S)$ the set of orbits of maximum rank that are smaller than $(v,S)$.
If it is not empty it admits a maximum.
Then the following theorem characterizes the Bruhat $\mathcal{G}$-order in $\mathcal{D}$
\begin{manualtheorem}{\ref{GorderB}}
Let $(Bvx_S,\gamma),(Bux_R,\tau)\in\mathcal{D}$.
Then $(Bvx_S,\gamma)\leq (Bux_R,\tau)$ if and only if $Bvx_S\leq Bux_R$ and one of the following is true:
\begin{enumerate}
\item both $\gamma$ and $\tau$ are trivial;
\item both $\gamma$ and $\tau$ are non-trivial;
\item $\#S\neq 2$ and $\gamma$ is trivial while $\tau$ is non-trivial;
\item $\#S=2$, $u<v$ and $\gamma$ is trivial while $\tau$ is non-trivial;
\item $\gamma$ is non-trivial while $\tau$ is trivial, $H(u,R)\neq \varnothing$ and $(u',R')=\max H(u,R)$ verifies
$(v,S)\leq (u',R')$ with $v<u'$.
\end{enumerate}
\end{manualtheorem}

For $\Phi$ of type \textbf{C} we will see that the number of non isomorphic local systems for an orbit $(v,S)$ is equal to $2^i$ where $i$ is the number of long roots in $S$.
In this case the characterization of the Bruhat $\mathcal{G}$-order is incomplete.
We can associate to any local system on $(v,S)$ a sequence $X(S)$ of length $i$ of $1$s and $-1$s.
Then, to any sequence $X$ of this kind we can apply an algorithm to reduce it to a sequence $\re(X)$ of smaller length where the $1$s and $-1$s are alternated.
The most important result for this case is the following:
\begin{manualtheorem}{\ref{GorderC}}
Let $X$ and $Y$ be the sequences associated respectively to a local system on $(v,S)$ and $(u,R)$.
Then the corresponding element in $\mathcal{D}$ are in the same connected component of the Hasse diagram if and only if $\re(X)=\re(Y)$.
\end{manualtheorem}
Note that this characterizes the connected components of the Hasse diagram and it gives a necessary condition for elements of $\mathcal{D}$ to be comparable.

The paper is organized the following way.
In section \ref{notation} we will introduce the notations and definitions we will use across all the paper.
In section \ref{orbits} there will be a brief overview of many results from \cite{GM} regarding the orbits in $G/L$ and the Bruhat order between them, while in section \ref{linebundle} we will follow \cite{Vogan} and \cite{LV} results on the local systems on the aforementioned orbit. 
Both these sections will introduce many useful theorems and properties that will be used in the latter part of the paper.
Finally sections \ref{sl},\ref{B} and \ref{C} will contain our results regarding respectively the simply laced case, the type \textbf{B} case and the type \textbf{C} case.

\textit{Acknowledgements}. I want to thank Jacopo Gandini and Andrea Maffei for their help in understanding the problem at hand, especially from the geometric point of view.
In particular, I want to thank Jacopo Gandini for his help with the more material stuff, like computing the connected components of the stabilizers in the simply connected and adjoint cases and teaching me how to use the software Lie to verify ideas and conjectures.

\section{Notations and definitions}\label{notation}

From now on, $G$ will be a connected semisimple algebraic group over $\mathbb{C}$.
We suppose that $G$ admits a parabolic subgroup $P$ such that the unipotent radical $P^u$ of $P$ is abelian.
This is the same as asking that Lie algebra $\mathfrak{p}^u$ of $P^u$ is abelian.

Fix $B$ a Borel subgroup of $G$ such that $B\subseteq P\subseteq G$ and a torus $T\subseteq B$.
This gives a root system $\Phi=\Phi(G,T)$ and a basis for $\Phi$ that we denote with $\Delta$.
The set of positive roots will be denoted with $\Phi^+$.
Recall that the Lie algebra $\mathfrak{g}$ of $G$ admits a decomposition
\[\mathfrak{g}=\mathfrak{t}\oplus \bigoplus_{\alpha\in\Phi}\mathfrak{u}_\alpha\]
where $\mathfrak{t}$ is the Lie algebra of $T$ and $\mathfrak{u}_\alpha$ is the root space relative to the root $\alpha$ which is always uni-dimensional.
Similarly, $G$ can be generated by $T$ and the images of the one parameter subgroups $u_\alpha\colon\mathbb{C}\longrightarrow G$ for every $\alpha\in \Phi$.
We denote the images of these morphisms with $U_\alpha$.

It is known that every parabolic group $Q\supseteq B$ corresponds to a subset of $\Delta$ and it is easy to see that our parabolic group $P$ must correspond to a subset $S=\Delta\setminus \left\lbrace\alpha_P\right\rbrace$ where $\alpha_P\in\Delta$ is a simple root that appears with coefficient $1$ in the decomposition of the highest root $\theta$.
Note that this implies that $\Phi$ can't be of type $\bf{E}_8,\bf{F}_4$ or $\bf{G}_2$.
We denote with $\Phi_P$ the root sub-system generated by the roots in $S$.
Put $\Psi=\Phi^+\setminus\Phi_P$.
Equivalently, $\Psi$ is the set of positive roots with $\alpha_P$ in their decomposition.

For every root $\alpha\in\Delta$ we have a hyperplane $H_\alpha$ in $\Phi\otimes \mathbb{R}$ which is orthogonal to $\alpha$ and the reflection which fixes $H_\alpha$ and sends $\alpha$ to $-\alpha$.
We will call these reflections \textit{simple} and we will denote them as $s_\alpha$.
The group of endomorphisms of $\Phi\otimes \mathbb{R}$ generated by the $s_\alpha (\alpha\in\Delta)$ is called the \textit{Weyl group} and we will denote it with $W$.
It is naturally isomorphic to $\N_G(T)/T$ where $\N_G(T)$ is the normalizer of $T$ in $G$.
Note that every element $w$ of $W$ admits a (non-unique) smallest expression as product of the $s_\alpha$ which we will call \textit{reduced expression} of $w$.
The length of this minimal expression will be the \textit{length} of $w$ and will be denoted as $l(w)$.
Note that if $\Phi^+(w)=\left\lbrace \alpha\in\Phi^+\mid v(\alpha)<0\right\rbrace$ then $l(w)=\#\Phi^+(w)$.

In $W$ we can consider the subgroup $W_P$ generated by the reflections $s_\alpha$ with $\alpha\neq \alpha_P$.
This is the Weyl group of $\Phi_P$.
Every coset of $W_P$ in $W$ admits a representative of minimal length; the set of these representatives is denoted with $W^P$.
We also have 
\[W^P=\left\lbrace w\in W\mid w(\alpha)>0 \forall \alpha\neq \alpha_P\right\rbrace\]

Now, $P$ admits a Levi decomposition $P\cong L\rtimes P^u$ where $L$ is called the  \textit{Levi subgroup} of $P$ and the variety $G/L$ is said to be a \textit{Hermitian symmetric variety}.
Note that $L$ is reductive and its root system is $\Phi_P$.
Put $B_L=B\cap L$.
The Borel subgroup acts on $G/L$ by multiplication and the orbits of this action will be the center of the next section.

While analysing the orbits we will come across some specific involutions in $W$.
In general, if $\sigma\in W$ is an involution and $\alpha\in\Delta$ we will say that $\alpha$ is:
\begin{enumerate}
\item \textit{real} if $\sigma(\alpha)=-\alpha$;
\item \textit{imaginary} if $\sigma(\alpha)=\alpha$;
\item \textit{complex} if $\sigma(\alpha)\neq \pm\alpha$.
\end{enumerate}

\section{The orbits}\label{orbits}

The first object of our study are the $B$-orbits in $G/L$.
Note that they naturally correspond to the $L$-orbits in the flag variety $B\backslash G$ as well as the $B\times L$ orbits in $G$.

It is easy to see that the (Zariski) closure of an orbit $\mathcal{O}$ is a union of $B$-orbits.
This let us put an order on the set of the orbits by imposing that $\mathcal{O}<\mathcal{O}'$ if and only if $\mathcal{O}\subseteq\cdr{\mathcal{O}'}$.
In analogy with the case of the $B$-orbits in $G/B$ this will be called the \textit{Bruhat order} on the $B$-orbits. 
Later, we will associate every one of these orbits with one or more non-isomorphic local systems.
The set $\mathcal{D}$ of pairs $(\mathcal{O},\gamma)$ where $\mathcal{O}$ is a $B$-orbit and $\gamma$ a (isomorphism class of) local system on $\mathcal{O}$ will be given another order in definition \ref{ordinelb}.
We will see that on the level of orbits this order is quite similar to the Bruhat order defined above.
Hence, the Bruhat order will be of great importance in our study.

The most complete result on the Bruhat order in this case can be found in \cite{GM}.
It equals the order among two orbits $\mathcal{O}$ and $\mathcal{O}'$ to the Bruhat order between combinatorial objects associated to these orbits.

In this paper we will use not only the final characterization of Gandini and Maffei, but also some of the intermediate results.
For this reason, we will now briefly summarize most of \cite{GM}.
Recall that $\Phi^+(v)=\left\lbrace \alpha\in\Phi^+\mid v(\alpha)<0\right\rbrace$

\begin{prop}[Lemma 2.4, \cite{GM}]\label{maxim}
Let $v\in W^P$ and let $\alpha\in \Delta$ such that $s_\alpha v<v$. 
Denote $\beta=-v^{-1}(\alpha)$.
Then $\beta$ is maximal in $\Phi^+(v)$ and minimal in $\Psi\setminus \Phi^+(s_\alpha v)$.

Vice versa:
\begin{enumerate}
\item if $\beta$ is maximal in $\Phi^+(v)$ then $\alpha=-v(\beta)\in \Delta$ and $s_\alpha v<v$;
\item if $\beta$ is minimal in $\Psi\setminus \Phi^+(v)$ then $\alpha=v(\beta)\in \Delta$ and $s_\alpha v>v$.
\end{enumerate}
\end{prop}

We denote with $<$ the Bruhat order on $W$.
We now study the Bruhat order between elements of $W^P$.
\begin{lemma}[Proposition 2.3, \cite{GM}]\label{ordWP}
Let $v,w\in W^P$.
Then $v\leq w$ if and only if $\Phi^+(v)\subseteq \Phi^+(w)$.
\end{lemma}

Note that this means that $v\in W^P$ is uniquely determined by $\Phi^+(v)$.
Moreover, it is clear that the sets $\Phi^+(v)$ are saturated in the sense that if $\beta\in \Phi^+(v)$ and $\alpha\leq \beta$, then $\alpha\in\Phi^+(v)$.

\begin{prop}
If $V\subseteq \Psi$ is saturated then there is $v\in W^P$ such that $V=\Phi^+(v)$.
\end{prop}
\begin{proof}
We will show this by induction on the cardinality of $V$.
If $\#V=0$ the claim is clear, so suppose $\#V=d>0$ and fix $\beta\in V$ maximal.
Then $V'=V\setminus\left\lbrace\beta\right\rbrace$ is still saturated, hence by induction there is $v'$ such that $V'=\Phi^+(v')$.
But now $\beta$ is minimal in $\Psi\setminus \Phi^+(v')$ so there is $\alpha\in \delta$ such that $V=\Phi^+(s_\alpha v')$ as we wanted.
\end{proof}

Note that this implies that for every $v,v'\in W^P$ the set $\left\lbrace w\in W^P\mid w\leq v, w\leq v'\right\rbrace$ has a maximum $w_0$ that is defined by $\Phi^+(w_0)=\Phi^+(v)\cap \Phi^+(v')$.
Moreover, the maximal elements of a saturated set $V\subseteq \Psi$ uniquely identify $V$, hence they uniquely identify $v\in W^P$.

\begin{lemma}
Let $v, w\in W^P$ and $v\leq w$.
Then there is a sequence of simple roots $\alpha_1,\ldots,\alpha_n$ such that $w=s_{\alpha_n}\cdots s_{\alpha_1}v$ and for every $i\in\left\lbrace 1,\ldots,n\right\rbrace, s_{\alpha_i}\left(s_{\alpha_{i-1}}\cdots s_{\alpha_1}v\right)>s_{\alpha_{i-1}}\cdots s_{\alpha_1}v$
\end{lemma}
\begin{proof}
We know that $v\leq w$ if and only if $\Phi^+(v)\subseteq \Phi^+(w)$.
Then, if $v\neq w$ there is a minimal element $\beta\in\Phi^+(w)\setminus\Phi^+(v)$.
This element must be minimal also in $\Psi\setminus \Phi^+(v)$, so there is $\alpha\in \Delta$ such that $v<s_\alpha v\leq w$ and inductively we conclude.
\end{proof}

For general elements in $W$ we have the following:

\begin{lemma}[Lemma 2.7, \cite{GM}]\label{propord}
Let $u,v\in W$ and suppose $u<v$.
For every $\alpha\in \Delta$ we have:
\begin{enumerate}
\item if $s_\alpha u>u$ and $s_\alpha v>v$ then $s_\alpha u<s_\alpha v$;
\item if $s_\alpha u<u$ and $s_\alpha v<v$ then $s_\alpha u<s_\alpha v$;
\item if $s_\alpha u>u$ and $s_\alpha v<v$ then $u\leq s_\alpha v$ and $s_\alpha u\leq v$.
\end{enumerate}
\end{lemma}


Following \cite{RS} we will associate to every orbit a particular involution in $W$.

Let $\mathcal{I}\subseteq W$ be the subset of all involutions.
We can define an action of the set of simple reflections $s_\alpha$ on $\mathcal{I}$ in the following way:
$$s_\alpha\circ\sigma =\left\lbrace \begin{array}{ll}
                  s_\alpha\sigma & \text{if } s_\alpha\sigma=\sigma s_\alpha\\
                  s_\alpha\sigma s_\alpha &\text{if } s_\alpha\sigma\neq\sigma s_\alpha\\
                  
                \end{array}
              \right.
$$
Note that $s_\alpha\circ\sigma=\tau$ if and only if $s_\alpha\circ\tau=\sigma$.
\begin{lemma}[3.1, \cite{GM}]\label{lemmaBruhatcirc}
Let $\alpha\in\Delta$ and $\sigma\in \mathcal{I}$.
Then $ s_\alpha \circ \sigma$ and $\sigma$ are always comparable.
Moreover, $s_\alpha \circ \sigma>\sigma$ if and only if $s_\alpha \sigma>\sigma$.
\end{lemma}
Note that if $s_\alpha\sigma\neq \sigma s_\alpha$ then $s_\alpha\sigma s_\alpha>s_\alpha\sigma >\sigma$ and $s_\alpha\sigma s_\alpha>\sigma s_\alpha >\sigma$.

The action on involutions interacts with the Bruhat orders with properties similar to the one in \ref{propord}

\begin{lemma}[3.2, \cite{GM}]\label{Bruhatcirc}
Let $\sigma,\tau\in \mathcal{I}$ and suppose $\sigma<\tau$.
For every $\alpha\in \Delta$ we have:
\begin{enumerate}
\item if $s_\alpha\circ\sigma>\sigma$ and $s_\alpha\circ\tau>\tau$ then $s_\alpha\circ\sigma<s_\alpha\circ\tau$;
\item if $s_\alpha\circ\sigma<\sigma$ and $s_\alpha\circ\tau<\tau$ then $s_\alpha\circ\sigma<s_\alpha\circ\tau$;
\item if $s_\alpha\circ\sigma>\sigma$ and $s_\alpha\circ\tau<\tau$ then $s_\alpha\circ\sigma\leq\tau$ and $\sigma\leq s_\alpha\circ \tau$.
\end{enumerate}
\end{lemma}
We define the \textit{length} of an involution $\sigma$ as
$$L(\sigma)=\frac{l(\sigma)+\lambda(\sigma)}{2}$$
where $l(\sigma)$ is the usual length in $W$ and $\lambda(\sigma)$ is the dimension of the $(-1)$-eigenspace of $\sigma$ on $\Phi\otimes \mathbb{R}$.

If $\sigma\in W$ is an involution, we know that $l(\sigma)$ is the length of a reduced expression for $\sigma$, that is, the minimum amount of simple reflections we need to compose to obtain $\sigma$.
As our intuition would suggest, this new length $L$ has the same property, except that instead of composing the simple reflections we use the action defined by $\circ$.

\begin{lemma}\label{invlength}
Let $\alpha\in\Delta$ and $\sigma\in \mathcal{I}$.

$$L(s_\alpha\circ\sigma)=\left\lbrace \begin{array}{ll}
										L(\sigma)+1 & \text{if }s_\alpha\circ\sigma>\sigma\\
										L(\sigma)-1 & \text{if }s_\alpha\circ\sigma<\sigma\\
									\end{array}
							\right.
$$
\end{lemma}



This and Lemma \ref{lemmaBruhatcirc} imply that if $\sigma\in W$ is an involution then $\sigma$ can be written as 
\[\sigma=s_{\alpha_1}\circ\ldots \circ s_{\alpha_{L(\sigma)}}\]
and $L(\sigma)$ is the minimum number with this property.

To every set $S\subseteq\Psi$ of mutually orthogonal roots we can naturally attach the involution

$$\sigma_S=\prod_{\alpha\in S} s_\alpha$$
Note that if $\alpha$ and $\beta$ are orthogonal then $s_\alpha s_\beta=s_\beta s_\alpha$, so $\sigma_S$ is well defined.
The $(-1)$-eigenspace of such involution is generated by $S$ so we have
$$L(\sigma_S)=\frac{l(\sigma_S)+\#S}{2}$$
\begin{lemma}[3.6, \cite{GM}]\label{sommaradici}
Let $\beta, \beta'\in \Psi$ be orthogonal.
Then:
\begin{enumerate}
\item $\beta$ and $\beta'$ are strongly orthogonal, that is $\beta\pm\beta'\notin\Psi$;
\item if $\beta+\alpha\in \Phi$ for some $\alpha\in \Phi^+$ then $\beta'+\alpha\notin \Phi$;
\item if $\beta-\alpha\in \Psi$ for some $\alpha\in \Phi^+$ then $\beta'-\alpha\notin \Psi$.
\end{enumerate}
\end{lemma}

Fix $S\subset \Phi$ orthogonal and define
\[\Gamma_S=\left\lbrace \alpha\in\Phi\mid \sigma_S(\alpha)=-\alpha\right\rbrace\]
With the notations of \cite{Vogan}, $\Gamma_S$ is the set of real descents of $\sigma_S$.
The following result is an easy corollary of proposition 3.8 of \cite{GM}, but we will state it here given its importance:

\begin{prop}\label{complinv}
Suppose that $\Phi$ is simply laced and $S\subseteq \Phi$ is strongly orthogonal.
Then
$\Gamma_S=S\cup-S$.
\end{prop}

Now, consider the projection map $\pi\colon G/L\longrightarrow G/P$.
It is $B$-equivariant.
Recall that $G/P=\bigcup_{v\in w^P} BvP/P$ and for $v\in W^P$ define $B^v=vPv^{-1}\cap B$ the stabilizer of $vP/P\in G/P$ in $B$.
Then
$$\pi^{-1}(BvP/P)= BvP/L\cong B\times^{B^v} \pi^{-1}(vP/P)=B\times^{B^v} vP/L$$
Hence we have a bijection between the $B$-orbits in $BvP/L$ and the $B^v$ orbits in $vP/L$ which is compatible with the Bruhat order.
If we define $B_v=P\cap v^{-1}Bv$ then these orbits are in bijection with the $B_v$-orbits in $P/L$.
\begin{lemma}(4.1, \cite{GM})
Let $v\in W^P$. 
Then $B_L=B_v\cap L$ and $B_v=B_L\ltimes U_v$ where $U_v$ is the subgroup of $P_u$ generated by the $U_\alpha$ with $\alpha\in \Psi\setminus\Phi^+(v)$.
\end{lemma}
Note that the Lie algebra of $U_v$ is $\mathfrak{u}_v=\bigoplus_{\alpha\in \Psi\setminus\Phi^+(v)} \mathfrak{u}_\alpha\subseteq \mathfrak{p}^u$.

Let $\exp\colon\mathfrak{p}_u\longrightarrow P^u$ be the exponential map and compose it with the projection $\pi\colon G\longrightarrow G/L$.
We obtain an isomorphism $r_P\colon\mathfrak{p}_u\longrightarrow P/L$ that is not $P$-equivariant if we consider the adjoint action on $\mathfrak{p}_u$ and the left multiplication on $P/L$.
We want to define an action of $P$ on $\mathfrak{p}_u$ that makes $r_P$ a $P$-equivariant map.
Consider the isomorphisms
$$L\ltimes \mathfrak{p}_u\cong L\ltimes P^u\cong P$$
from left to right $(g,y)\longmapsto g\exp(y)$.
Note that with this identification we have $B_v=B_L\ltimes \mathfrak{u}_v$.
Let $(g,y)\in P$ and $x\in\mathfrak{p}_u$. 
Define the action

\begin{equation}\label{action}
(g,y).x=\adj_g(x+y)
\end{equation}
\begin{lemma}[4.2, \cite{GM}]\label{orderiso}
Let $v\in W^P$.
Then the map $B_ve\longmapsto Bv\exp(e)L/L$ is an order isomorphism between the $B_v$-orbits in $\mathfrak{p}_u$ and the $B$-orbits in $BvP/L$.
\end{lemma}

It follows that if we want to parametrize the $B$-orbits on $G/L$ it is enough to parametrize the $B_v$-orbits in $\mathfrak{p}^u$ for every $v\in W^P$.
This is easier to handle because we know very well the adjoint representation of $G$ over $\mathfrak{g}$.
It turns out (Proposition 4.7, \cite{GM}) that these are parametrized exactly by the orthogonal subsets in $\Phi^+(v)$.
Hence, the $B$-orbits in $G/L$ are parametrized by the following combinatorial objects.

\begin{defin}[Admissible pairs]
Let $(v,S)$ be a pair such that $v\in W^P$ and $S\subseteq \Psi$.
Then it is \textit{admissible} if and only if $S$ is orthogonal and $S\subseteq \Phi^+(v)$.
\end{defin}

We will say that $v$ is the $W^P$-part and $S$ is the $\Psi$-part of $(v,S)$.

As we said, the admissible pairs parametrize the orbits.
The following theorem is Corollary 4.8 in \cite{GM}.
Note that following \cite{GM} we denote $x_S=\exp(e_S)L$ for every $S\subseteq \Psi$.

\begin{theorem}
There is a correspondence:
\begin{align*}
\left\lbrace \text{admissible pairs} \right\rbrace &\longrightarrow \left\lbrace B-\text{orbits in }G/L\right\rbrace\\
(v,S)&\mapsto Bvx_S
\end{align*}
\end{theorem}


With this parametrization comes a combinatorial characterization of the orbits, also in \cite{GM}.
If $w\in W$ we will denote with $\left[w\right]^P$ the minimal length representative of the coset $wW_P$.
\begin{theorem}\label{ordineGM}
Let $(v,S)$ and $(u,R)$ be admissible pairs.
Then $Bux_R<Bvx_S$ if and only if 
\[ \left[v\sigma_S\right]^P\leq \left[u\sigma_R\right]^P\leq u\leq v \text{ and }\sigma_{u(R)}\leq\sigma_{v(S)}\]
\end{theorem}

The proof of Maffei and Gandini makes great use of the action of the minimal parabolic groups which we will now introduce.

Given a simple root $\alpha\in \Delta$ we can define a parabolic subgroup $P_\alpha$ which is the subgroup generated by $B$ and $U_{-\alpha}$.
It is minimal among the parabolic subgroups that strictly contain $B$ and every such subgroup is obtained this way.

Now fix a $B$-orbit $BxL/L$ in $G/L$ and a simple root $\alpha\in\Delta$.
The minimal parabolic subgroup $P_\alpha$ acts on $G/L$, so the $B$-orbit $BxL/L$ is contained in the $P_\alpha$-orbit $P_\alpha xL/L$.
\begin{prop}
The Borel subgroup $B$ acts on $P_\alpha xL/L$ with finitely many orbits.
In fact there are at most $3$ $B$-orbits in $P_\alpha xL/L$.
\end{prop}


There must be a unique $B$-orbit $\mathcal{O}$ in $P_\alpha vx_S$ such that $\cdr{\mathcal{O}}=P_\alpha vx_S$.
We will call $\mathcal{O}$ the \textit{open orbit} of $P_\alpha vx_S$.

We also have
\[\dim P_\alpha=\dim B+1\]
so $\dim Bvx_S\leq \dim P_\alpha vx_S\leq \dim Bvx_S+1$.
In particular if $\mathcal{O}$ is the open orbit, then $\dim \mathcal{O}=\dim P_\alpha vx_S$.
This implies that if $\mathcal{O}$ and $\mathcal{O}'$ are distinct $B$-orbits in $P_\alpha vx_S$ then they are comparable if and only if one of them is the open orbit.

Following \cite{GM}, we will use this notation:
\begin{defin}
Let $(v,S)$ be an admissible pair and $\alpha\in \Delta$ a simple root.
Then we define
\begin{enumerate}

\item $m_\alpha(v,S)=(v',S') \text{ if and only if }(v',S') \text{ is the open orbit in }P_\alpha vx_S$;
\item $m(\alpha_1)(v,S)=m_{\alpha_1}(v,S)$ and inductively $m(\alpha_1,\ldots,\alpha_n)(v,S)=m_{\alpha_n}m(\alpha_1,\ldots,\alpha_{n-1})(v,S)$;
\item $\mathcal{E}_\alpha(v,S)=\left\lbrace (v',S')\neq (v,S)\mid m_\alpha(v',S')=(v,S)\right\rbrace$.
\end{enumerate}
\end{defin}

We will also say that $\alpha\in\Delta$ is an \textit{ascent} for $(v,S)$ if $m_\alpha(v,S)\neq (v,S)$ and that it is a \textit{descent} if $\mathcal{E}_\alpha(v,S)\neq 0$.
We will similarly say that $\alpha$ is an \textit{ascent} for $\sigma_{v(S)}$ if $\sigma_{v(S)}(\alpha)>0$ and that is a \textit{descent} if $\sigma_{v(S)}(\alpha)<0$.

We will also say that a sequence $\alpha_1\ldots,\alpha_n$ in $\Delta$ is a \textit{sequence of ascents} for $(v,S)$ if $\alpha_1$ is an ascent for $(v,S)$ and for every $i=2,\ldots, n$ we have that $\alpha_i$ is an ascent for $m(\alpha_1,\ldots,\alpha_{i-1})(v,S)$.

Similarly, we will say that $\alpha_1,\ldots,\alpha_n$ is a \textit{sequence of descents} for $(v,S)$ if there is a sequence of orbits $\mathcal{O}_1,\ldots,\mathcal{O}_n=(v,S)$ such that $\alpha_i$ is a descent for $\mathcal{O}_i$ and $\mathcal{O}_{i-1}\in\mathcal{E}_{\alpha_i}\mathcal{O}_i$ for every $i=2,\ldots,n$.

Now, we can put an order in the set of orbits by imposing that $\mathcal{O}\leq m_\alpha \mathcal{O}$ and that if $\mathcal{O}\leq\mathcal{O}'$, then $m_\alpha\mathcal{O}\leq m_\alpha \mathcal{O}'$.
The smallest order with this property is called the \textit{standard order} in \cite{RS} and it is in fact equivalent to the Bruhat order.
It follows that the Bruhat order admits the following characterization.

\begin{lemma}\label{storder}
We have $Bvx_S\leq Bux_R$  in the Bruhat order if and only if there is a sequence $Bvx_S=\mathcal{O}_1,\mathcal{O}_2,\ldots,\mathcal{O}_n=Bux_R$ of orbits such that for every $i=1,\ldots,n-1$ there is $k\in\mathbb{N}$, a sequence $(\alpha_{i,0},\alpha_{i,1},\ldots,\alpha_{i,k})\in\Delta$ and an orbit $\mathcal{U}_i$ with the following properties:
\begin{enumerate}
\item $\alpha_{i,0}$ is an ascent for $\mathcal{U}_i$ and $(\alpha_{i,1},\ldots,\alpha_{i,k})$ is a sequence of ascents both for $\mathcal{U}_i$ and $m_{\alpha_{i,0}}\mathcal{U}_i$;
\item $\mathcal{O}_i=m(\alpha_{i,1},\ldots,\alpha_{i,k})\mathcal{U}_i$ and $\mathcal{O}_{i+1}=m(\alpha_{i_0},\alpha_{i,1},\ldots,\alpha_{i,k})\mathcal{U}_i$.
\end{enumerate}
\end{lemma}

In the proof of Theorem \ref{ordineGM} by Gandini and Maffei as well as in the rest of this paper the following lemma is fundamental:

\begin{lemma}[Lemma 5.1, \cite{GM}]
Let $(v,S)$ be an admissible pair and $\alpha\in\Delta$.
Then
\begin{enumerate}
\item if $m_\alpha(v,S)=(v',S')\neq (v,S)$ then $\sigma_{v'(S')}=s_\alpha\circ\sigma_{v(S)}>\sigma_{v(S)}$;
\item $\mathcal{E}_\alpha(v,S)\neq \varnothing$ if and only if $s_\alpha\circ\sigma_{v(S)}<\sigma_{v(S)}$ or equivalently $\sigma_{v(S)}(\alpha)<0$.
\end{enumerate}
\end{lemma}

\section{The local systems}\label{linebundle}

Following Lusztig and Vogan (\cite{Vogan} and \cite{LV}) we can consider for every $B$-orbit $\mathcal{O}$ in $G/L$ a $B$-equivariant local system over $\mathcal{O}$.
We will use the following definition:
\begin{defin}[$B$-equivariant $\mathbb{C}$-local system of rank 1]
A $\mathbb{C}$-local system of rank $1$ (or simply \textit{local system}) over $\mathcal{O}$ is a complex line bundle $\pi\colon E\longrightarrow \mathcal{O}$ with constant transition functions such that for every trivializing open $U\subseteq \mathcal{O}$ we have $\pi^{-1}(U)\cong U\times \mathbb{C}$ where $\mathbb{C}$ has the discrete topology.

We say that a local system over $\mathcal{O}$ is \textit{$B$-equivariant} if there is an action $B\times E\longrightarrow E$ of $B$ on $E$ that makes the following diagram commute:
\begin{center}
\begin{tikzcd}
B\times E\arrow[d,"\pi"] \arrow[r] & E\arrow[d,"\pi"]\\
B\times \mathcal{O}\arrow[r] & \mathcal{O}
\end{tikzcd}
\end{center}
\end{defin}
It is easy to see that the isomorphism classes of local systems of this kind are in a one to one correspondence with the continuous representations of $\stab_B(x)$ on the stalk at $x$ for any one $x\in\mathcal{O}$.
In this case we consider in $\mathbb{C}$ the discrete topology, so a continuous representation is a representation of $\pi_0(\stab_B(x))$, the group of connected components.

\begin{defin}
For a $B$-orbit $\mathcal{O}$ define $L_\mathcal{O}$ to be the set of isomorphism classes of $B$-equivariant local systems over $\mathcal{O}$. 
Then we define 
\[\mathcal{D}=\left\lbrace (\mathcal{O},\gamma)\mid \mathcal{O} \text{ is a } B\text{-orbit in }G/L, \gamma\in L_\mathcal{O}\right\rbrace\]

\end{defin}

Note that the information on the underlying orbit is inherently contained in the line bundle, so we will often refer to $(\mathcal{O},\gamma)\in\mathcal{D}$ simply as $\gamma$.

We know by now that $\Delta$ acts on the set of the orbits through the minimal parabolic subgroups.
Fix an orbit $Bvx_S$ and $\alpha\in\Delta$.
If $\gamma$ is a (isomorphism class of) local system on $Bvx_S$ we can ask when and how we can extend $\gamma$ to $P_\alpha vx_S$.
If $(Bvx_S,\gamma)=\gamma\in\mathcal{D}$ define
\[\alpha\circ\gamma=\left\lbrace (\mathcal{O},\tau)\mid \mathcal{O}=m_\alpha(v,S)\neq (v,S) \text{ and }\tau \text{ extends }\gamma \text{ to }Bvx_S\cup\mathcal{O}\right\rbrace\]

By combining definition 6.4 of \cite{Vogan} and Lemma 3.5 of \cite{LV} we get the following.

\begin{lemma}\label{listofaction}
Fix $\alpha\in\Delta$ and $(v,S)$ an admissible pair with local system $\gamma$.
Then one and only one of the following is true:
\begin{description}
\item[a)] $\alpha$ is imaginary for $\sigma_{v(S)}$ and $P_\alpha vx_S=Bvx_S$.
We say in this case that $\alpha$ is \textit{compact imaginary} for $(v,S)$;
\item[b1)] $\sigma_{v(S)}(\alpha)>0$ and $\sigma_{v(S)}(\alpha)\neq \alpha$. 
Then $m_\alpha(v,S)=(u,R)\neq (v,S)$ and $\alpha\circ \gamma$ contains a single element;
\item[b2)] $\sigma_{v(S)}(\alpha)<0$ and $\sigma_{v(S)}(\alpha)\neq -\alpha$.
Then $\mathcal{E}_\alpha(v,S)$ contains a single orbit $(u,R)$ and there is a unique $\gamma'$ over $(u,R)$ such that $\alpha\circ \gamma'=\left\lbrace \gamma\right\rbrace$;
\item[c1)] $\sigma_{v(S)}(\alpha)=\alpha$ and $P_\alpha v x_S$ contains only two orbits.
Then $m_\alpha(v,S)=(u,R)\neq (v,S)$ and $\alpha\circ \gamma$ contains exactly two elements;
\item[c2)] $\sigma_{v(S)}(\alpha)=-\alpha$, $\mathcal{E}_\alpha(v,S)$ contains only an orbit $(u,R)$ and there is $\gamma'$ over $(u,R)$ such that $\gamma\in\alpha\circ \gamma'$.
Then $\gamma'$ is unique and $\alpha\circ\gamma'$ contains exactly two elements (one being $\gamma$);
\item[d1)] $\sigma_{v(S)}(\alpha)=\alpha$ and $P_\alpha vx_S$ contains three orbit.
Then $m_\alpha(v,S)=(u,R)\neq (v,S)$ and $\alpha\circ \gamma$ contains a single element;
\item[d2)] $\sigma_{v(S)}(\alpha)=-\alpha$ and $P_\alpha vx_S$ contains three orbits and $\gamma$ extends to all $P_\alpha vx_S$.
Then this extension is unique;
\item[e)]  $\sigma_{v(S)}(\alpha)=-\alpha$.
Then $\mathcal{E}_\alpha(v,S)\neq \varnothing$ but $\gamma$ could admit no extensions.

\end{description}

\end{lemma}

There are two crucial results in the above lemma that we need to emphasize.
The first one, is that every time we have an ascent $\alpha$ for $(v,S)$ and a (isomorphism class of) local system $\gamma$ on $(v,S)$, then $\gamma$ always extends to $P_\alpha vx_S$ and $\alpha\circ\left((v,S),\gamma\right)\neq \varnothing$.
The second is that this extension is unique except in case $c1)$.
It is then logical to study this case with more attention.
For this recall that $P_\alpha=B\cup Bs_\alpha B=Bs_\alpha \cup BU_{-\alpha}$.

Fix $(v,S)$ and $\alpha\in\Delta$ with $\sigma_{v(S)}(\alpha)=\alpha$.
Put $\beta=v^{-1}(\alpha)$.
If $\beta\in\Psi$, then $P_\alpha vx_S$ contains three orbits, so that's excluded.
Then $\beta\in\Delta$ and
\[P_\alpha vx_S=Bvx_{s_\beta(S)}\cup Bvx_S\cup\bigcup_{t\in\mathbb{C}^*} Bvu_{-\beta}(t)x_S\]
Note that
\[\sigma_{v(S)}(\alpha)=\alpha\Leftrightarrow \sigma_{S}(\beta)=\beta\Rightarrow \sigma_Ss_\beta=s_\beta\sigma_S\Leftrightarrow s_\beta\sigma_Ss_\beta=\sigma_S\Leftrightarrow \sigma_{s_\beta(S)}=\sigma_S\]
and the last equality implies $s_\beta(S)=S$ because they are both orthogonal subsets of $\Phi^+(v)$.
Moreover, suppose that the set $H=\left\lbrace \gamma\in S\mid (\gamma,\beta)\neq 0\right\rbrace$ is non empty.
Then 
\[\beta=\sigma_S(\beta)=\beta+\sum_{\gamma_i\in H} a_i\gamma_i\]
with $a_i\neq 0$.
But that's absurd, because $\gamma_i$ are linearly independent, so $H=\varnothing$.

For a fixed $t\in\mathbb{C}^*$ we know that $Bvu_{-\beta}(t)x_S=Bv\exp(u_{-\beta}(t).e_S)$.
There must then be $\gamma\in S$ such that $\gamma-\beta\in\Psi$.
But, $(\gamma,\beta)=0$, so it must also be $\gamma+\beta\in\Psi$.
It follows that we can be in case $c1)$ if and only if there is a root $\gamma\in S$ to which we can both add and subtract $\beta$.
This simple fact will be useful later.

Now we look back at $\mathcal{D}$.
Following \cite{Vogan} and \cite{LV} we will endow it with an order.

\begin{defin}[Bruhat $\mathcal{G}$-order]\label{ordinelb}
We call Bruhat $\mathcal{G}$-order and we denote it with $<$ the smallest order on $\mathcal{D}$ with the following properties:
\begin{enumerate}
\item if $\gamma'\in\alpha\circ\gamma$, then $\gamma<\gamma'$;
\item if $\gamma<\tau$ and $\gamma'\in\alpha\circ\gamma,\tau'\in\alpha\circ\tau$ then $\gamma'\leq\tau'$.
\end{enumerate}
\end{defin}

\begin{prop}\label{propordlb}
Fix $\gamma\leq\tau\in\mathcal{D}$ and suppose there is a sequence $(\alpha_1,\ldots,\alpha_n)$ in $\Delta$ with a subsequence $(\alpha_{i_1},\ldots,\alpha_{i_k})$ such that:
\begin{enumerate}
\item there is a sequence $(\tau_1,\ldots,\tau_n)$ such that $\tau_i\in\alpha_i\circ\tau_{i-1}$ where $\tau_0=\tau$;
\item there is a sequence $(\gamma_1,\ldots,\gamma_k)$ such that $\gamma_h=\alpha_{i_h}\circ\gamma_{h-1}$ where $\gamma_0=\gamma$.
\end{enumerate}
Then $\gamma_k\leq \tau_n$.
\end{prop}
\begin{proof}
We will show this by induction on $n$.
If $n=1$ and $k=0$ it is clear by property $1$ above and the transitive property.
If $n=k=1$ the claim is property $2$ above.

Now suppose $n>1$.
If $i_k=n$, then we know by inductive hypothesis that $\gamma_{k-1}\leq \tau_{n-1}$ so we conclude with property $2$.
If instead $i_k<n$ we know $\gamma_k\leq \tau_{n-1}$ and we conclude with property $1$.
\end{proof}

There is an evident similarity between the Bruhat $\mathcal{G}$-order and the standard order of \cite{RS}.
It is proved in \cite{RS} that the standard order is equivalent to the Bruhat order, so it is not unexpected that there is a relation between the Bruhat $\mathcal{G}$-order and the Bruhat order.

\begin{prop}[Lemma 5.9, \cite{Vogan}]
Suppose $(Bvx_S,\gamma),(Bux_R,\tau)\in\mathcal{D}$ and $(Bvx_S,\gamma)\leq (Bux_R,\tau)$.
Then $Bvx_S\leq Bux_R$.
\end{prop}

Every orbit admits at least an isomorphism class of local systems; the trivial one.
So, we can consider the subset 
\[\mathcal{D}_0=\left\lbrace (Bvx_S,\gamma)\in\mathcal{D}\mid \gamma \text{ is trivial}\right\rbrace\]
which is clearly in one to one correspondence with the set of orbits on $G/L$.
We can now ask what is the relation between the Bruhat $\mathcal{G}$-order restricted to $\mathcal{D}_0$ and the Bruhat order on the set of orbits.
The answer is that they are the same.

\begin{prop}\label{ordtriv}
Let $(Bvx_S,\gamma),(Bux_R,\tau)\in\mathcal{D}_0$.
Then $Bvx_S\leq Bux_R$ if and only if 
\[(Bvx_S,\gamma)\leq (Bux_R,\tau)\]
\end{prop}
\begin{proof}
Given that the local systems are trivial, they always admit extensions and given that we restricted ourselves to the trivial local systems, this extensions are always unique.

Note that one implication is just the proposition above, so suppose $Bvx_S\leq Bux_R$.
By \cite{RS} the Bruhat order is equivalent to the standard order and it can be characterized as in Lemma \ref{storder}.
So it is enough to prove the claim for $n=2$.
But this is clear by Proposition \ref{propordlb}.
\end{proof}

In the next sections we will always consider $\Phi$ to be irreducible and the algebraic group $G$ to be simply connected.
The second assumption in particular seems quite restrictive.
Actually, it is a matter of tedious but easy calculations to show that in some cases when $G$ is simply connected the orbits admit non-trivial local systems and we will give proof of this in the following chapter.
It is also easy to see that, on the other hand, when $G$ is adjoint the orbits never admit non-trivial local systems and so our question is easily answered by Proposition \ref{ordtriv}.
\section{The simply laced case}\label{sl}
In this section we will suppose that the root system $\Phi$ is simply laced, which means it is of type $\bf{ADE}$.

In the first part we will show a result on the Bruhat order of a specific subset of $B$-orbits in $G/L$, which are the 
orbits of maximum rank.
While the result concerns the geometry of the orbits, most of the proof will be purely about the combinatorial properties of the involutions associated to these orbits.

In the last part we will show that the orbits of maximum rank are exactly the orbits which can have non-trivial local systems and  thanks to the preceding results  we will show that the Hasse diagram has two connected components and that in every connected component the order induced by the local systems coincides with the Bruhat order on the underlying orbits.

We will start with a formal definition of orbit of maximum rank.
\begin{defin}
A set of orthogonal roots $S\subseteq \Psi$ is said to be of \textit{maximum rank} in $\Psi$ if for every $\alpha\in \Psi$ either $\alpha\in S$ or $S\cup\left\lbrace \alpha\right\rbrace$ is not orthogonal.

Having fixed $P$ and $\Psi$, an admissible pair $(v,S)$ will be of \textit{maximum rank} if and only if $S$ is of maximum rank in $\Psi$.
The orbit associated to such pair will also be an orbit of \textit{maximum rank}.

Finally, we will denote with $\RM$ the set of admissible pairs (or orbits) of maximum rank.
\end{defin}

Note that a set $S\subseteq \Psi$ of maximum rank need not to be of maximum rank in all $\Phi$.
We will also introduce some additional notation regarding descents and ascents with respect to the action of minimal parabolic subgroups.

\begin{defin}
Suppose that $(v,S)$ is admissible and $\alpha$ is a descent for $\sigma_{v(S)}$. Let $\beta=v^{-1}(\alpha)$.
Then we have the following possibilities:
\begin{itemize}
\item $\alpha$ is real. Then $0<-\beta\in S$ and $s_\alpha v<v$. 
The other $B$-orbits in the $P_\alpha$-orbits of $(v,S)$ are $(s_\alpha v,S\setminus\left\lbrace-\beta\right\rbrace)$ and $(v,S\setminus\left\lbrace-\beta\right\rbrace)$.
They both share the same involution $s_\alpha\sigma_{v(S)}$. 
In this case we will say that $\alpha$ is \textit{real};
\item $\alpha$ is complex and $0<-\beta\in\Psi$. Then $-\beta\notin S$ and $s_\alpha v<v$.
In this case the other $B$-orbit is $(s_\alpha v,S)$ with involution $s_\alpha\sigma_{v(S)}s_\alpha$. We will say that $\alpha$ is a \textit{descent on v} or \textit{on $W^P$};
\item $\alpha$ is complex and $\beta\in \Delta$.
Then the other $B$-orbit is $(v,s_\beta(S))$ with involution $s_\alpha\sigma_{v(S)}s_\alpha$.
We will say that $\alpha$ is a \textit{descent on S} or \textit{on $\Psi$}.

\end{itemize}
\end{defin}

We have similar definitions for the ascents.
\begin{defin}
Suppose that $(v,S)$ is admissible and $\alpha$ is an ascent for $\sigma_{v(S)}$. Let $\beta=v^{-1}(\alpha)$.
Then we have the following possibilities:
\begin{itemize}
\item $\alpha$ is imaginary and $P_\alpha vx_S=Bvx_S$.
Then we say that $\alpha$ is \textit{compact imaginary};
\item $\alpha$ is imaginary and $P_\alpha vx_S\neq Bvx_S$. 
Then, if $v'$ is the maximum element between $v$ and $s_\alpha v$, the open $B$-orbit in $P_\alpha vx_S$ is $(v',S\cup \left\lbrace \beta\right\rbrace)$ with involution $s_\alpha\sigma_{v(S)}$. 
In this case we will say that $\alpha$ is \textit{non-compact imaginary};
\item $\alpha$ is complex and $0<\beta\in\Psi$. Then $\beta\notin S$ and $s_\alpha v>v$.
In this case the open $B$-orbit is $(s_\alpha v,S)$ with involution $s_\alpha\sigma_{v(S)}s_\alpha$. We will say that $\alpha$ is an \textit{ascent on v} or \textit{on $W^P$};
\item $\alpha$ is complex and $\beta\in \Delta$.
Then the other $B$-orbit is $(v,s_\beta(S))$ with involution $s_\alpha\sigma_{v(S)}s_\alpha$.
We will say that $\alpha$ is an \textit{ascent on S} or \textit{on $\Psi$}.

\end{itemize}
\end{defin}

We need an additional property on the set $\Psi$.
Recall that on $\Phi$ we have the partial order $\alpha\leq\beta$ if and only if $\beta-\alpha$ can be written as a positive sum of roots in $\Delta$.
We say that a subset $S\subseteq \Phi$ is incomparable if $\alpha$ and $\beta$ are incomparable for every $\alpha,\beta\in S$.
\begin{property}\label{unic}
If $S,T\subseteq \Psi$ are of maximum rank and both all the roots in $S$ and all the roots in $T$ are incomparable then $S=T$.
\end{property}
Note that not every $\Psi$ has this property, but we will show later that all $\Psi$ for which non-trivial local systems exist have this property.
If $\Psi$ doesn't have property \ref{unic}, then most of the following results are false.
\begin{lemma}\label{due}
Suppose that $\Psi$ has property \ref{unic}.
Then 
\begin{enumerate}
\item if $S$ is of maximum rank and $\beta\in S$, the only set of roots of maximum rank that contains $S\setminus\left\lbrace\beta\right\rbrace$ is $S$;
\item fix $\alpha\in \Delta$ and $(v,S)\in\RM$.
Then either $\alpha$ is orthogonal to $v(S)$, $-\alpha\in v(S)$ or $\alpha$ is not orthogonal to exactly two roots in $\sigma_{v(S)}$.
\end{enumerate}
\end{lemma}
\begin{proof}
\begin{enumerate}
\item
Suppose that $S$ is of maximum rank and all the roots in $S$ are incomparable.
Suppose that there is $\beta\in S$ and $\gamma\in\Psi$ such that $\left(S\setminus\left\lbrace\beta\right\rbrace\right)\cup\left\lbrace\gamma\right\rbrace$ is orthogonal.
Then by maximality $(\beta,\gamma)\neq 0$, hence $\gamma-\beta\in\Phi^+$ or $\beta-\gamma\in\Phi^+$.
Moreover, by property \ref{unic} there must be $\alpha\in S$ such that $\alpha$ and $\gamma$ are comparable.

To simplify the notation, suppose $\tau=\gamma-\beta\in\Phi^+$, so $\gamma>\beta$.
If $\beta-\gamma\in\Phi^+$ the proof is similar.

It must be $\gamma>\alpha$, so $\gamma-\alpha=\sum a_i\alpha_i$ where $a_i\in\mathbb{N}$ and $\alpha_i\in \Delta$.
Suppose at first $\tau\in \Delta$.
By incomparability $\beta-\alpha$ is not a positive sum of simple roots.
But $\beta-\alpha=s_\tau(\gamma-\alpha)$ and $s_\tau$ changes the positivity only to $\tau$, hence $\tau$ must appear in $\gamma-\alpha$ which then implies again that $\beta-\alpha=\left(\gamma-\tau\right)-\alpha$ is positive.

Now consider the set 
\[ \left\lbrace \gamma\in \Psi\setminus S\mid \exists\beta\in S \text{ such that }\left(S\setminus\left\lbrace\beta\right\rbrace\right)\cup\left\lbrace\gamma \right\rbrace\text{ is orthogonal}\right\rbrace\]
This is not empty, so take a minimal element $\gamma$.

Then there are $\tau_1,\ldots,\tau_n\in \Delta$ ($n>1$) such that $\gamma=\beta+\tau_1+\cdots+\tau_n$ and for every $1\leq i\leq n$, $\beta+\tau_1+\cdots+\tau_i\in\Psi$.
Consider $\gamma'=\beta+\tau_1+\cdots+\tau_{n-1}$.
By the minimality of $\gamma$, there must be two roots $\alpha_1,\alpha_2\in S$ such that $\gamma'$ is not orthogonal to both.
In particular, it must be  $(\gamma',\alpha_1)>0$ and $(\gamma',\alpha_2)>0$.
Then $(\gamma,\alpha_1)=(\gamma'+\tau_n,\alpha_1)=(\gamma',\alpha_1)+(\tau_n,\alpha_1)=0$ and $(\gamma,\alpha_2)=(\gamma'+\tau_n,\alpha_2)=(\gamma',\alpha_2)+(\tau_n,\alpha_2)=0$ would imply $(\tau_n,\alpha_1),(\tau_n,\alpha_2)<0$ which is absurd because $\tau_n$ can't be added to two different roots in $S$.

We obtain the general result by noting that the Weyl group is transitive on the set of orthogonal roots with the same cardinality(\cite{RRS}).

\item Denote $\beta=v^{-1}(\alpha)$.
We know that $\beta\in\Delta_P$ or $\beta\in\Psi$ or $-\beta\in \Psi$.

In the first case, if $\beta$ is not orthogonal to $S$, then there is $\gamma\in S$ such that $(\gamma,\beta)\neq 0$.
Suppose for simplicity that $(\gamma,\beta)>0$, hence $\gamma-\beta\in \Psi$.
Then by the point above, there must be $\gamma'\neq \gamma\in S$ that is not orthogonal to $\gamma-\beta$, that is $(\gamma',\beta)\neq 0$.
Because of the strong orthogonality property we know that $(\gamma',\beta)< 0$ and that all the other roots in $S$ are orthogonal to $\beta$.

If $\beta\in \Psi$, then it must be $\beta\notin \Phi^+(v)$, so $\beta\notin S$.
Then, by maximality we know that there is a root in $S$ that is not orthogonal to $\beta$ and because of property above they must be at least two.
Suppose there are three of these roots: $\gamma_1,\gamma_2,\gamma_3$.
Then $\beta-\gamma_1\in\Phi$ and $(\beta-\gamma_1,\gamma_2)>0$, hence $\beta-\gamma_1-\gamma_2\in\Phi$.
But now $(\beta-\gamma_1-\gamma_2,\gamma_3)>0$, hence $\beta-\gamma_1-\gamma_2-\gamma_3\in\Phi$ and that's impossible because $[\beta-\gamma_1-\gamma_2-\gamma_3,\alpha_P]=-2$ and we know that for every root $\tau\in\Phi$ it must be $[\tau,\alpha_P]\in\left\lbrace -1,0,1\right\rbrace$.

If $-\beta\in \Psi$ then either $-\beta\in S$ or $-\beta\notin S$.
In the second case, repeat the reasoning above with $-\beta$ instead of $\beta$.
\end{enumerate}
\end{proof}
\begin{lemma}\label{prop} The maximum rank orbits have the following properties:
\begin{enumerate}
\item if $(v,S)\in \RM$ and $\alpha$ is an ascent for $(v,S)$ then $m_\alpha(v,S)=(u,R)\in\RM$ and $\#S=\#R$;
\item all the $\Psi$-parts of orbits in $\RM$ have the same cardinality;
\item there is a minimum orbit $(v_0,S_0)$. 
Moreover, $v_0$ is minimum among the $W^P$-parts of the orbits in $\RM$ and all the roots in $S_0$ are incomparable.
The involution $\sigma_{v_0(S_0)}$ is the product of $\#S_0$ simple reflections related to orthogonal roots;
\item if $\alpha$ is a complex descent for $(v,S)\in\RM$ then there are only two orbits in $P_\alpha vx_S$; $Bvx_S$ is the open one and the other is still in $\RM$;
\item \label{seq}  if $(v,S)\in \RM$ there is a sequence $(\alpha_n,\ldots,\alpha_1)\in\Delta^n$ of ascents for $(v_0,S_0)$ such that 
\[m(\alpha_n,\ldots,\alpha_1)(v_0,S_0)=(v,S)\]
This implies that
\[\sigma_{v(S)}=s_{\alpha_1}\cdots s_{\alpha_n}\sigma_{v_0(S_0)}s_{\alpha_n}\cdots s_{\alpha_1}\]
\item if $(v,S),(u,T)\in\RM$ and $\sigma_{v(S)}=\sigma_{u(T)}$ then $(u,T)=(v,S)$;
\item if $(v,S)\in \RM$ and $\alpha\in\Delta$ is such that $\beta=v^{-1}(\alpha)\in\Psi$, then $\alpha$ is a complex ascent (on $v$) for $(v,S)$ while if $\beta\in -\Psi$ it is a descent.
Moreover, if $-\beta\notin S$, the descent is complex (on $v$).
\end{enumerate}
\end{lemma}
\begin{proof}
\begin{enumerate}
\item Note that $\alpha$ can't be non-compact imaginary because $S$ is of maximum rank.
It follows that $\alpha$ is a complex ascent which is either an ascent on $v$ or on $S$.
In the first case $S=R$, so the claim is clear.
In the second case $R=s_\beta(S)$ for $\beta=v^{-1}(\alpha)$ which implies $\#S=\#R$ and also $S$ of maximum rank.

\item We know that there is an open orbit $(\omega_P,S)$ which is greater than any other orbit.
Note that any orbit $(v,T)\neq (\omega_P,S)$ admits at least an ascent.
It follows that there is a chain of ascents between $(v,T)$ and $(\omega_P,S)$ and by the above point $\#T=\#S$.

\item fix $(v,S)\in\RM$.
We want to show that $\sigma_{v(S)}$ is either the product of simple reflections related to orthogonal roots or admits a complex descents.
So, suppose that for every $\alpha\in \Delta$, $\sigma_{v(S)}(\alpha)$ is either $-\alpha$ or positive and denote $T=\left\lbrace \alpha\in\Delta \mid \sigma_{v(S)}(\alpha)=-\alpha\right\rbrace$.
By Proposition \ref{complinv} it must be $T\subseteq -v(S)$.
It follows that the roots in $T$ are mutually orthogonal.
But then $\sigma_T=\sigma_{v(S)}$.
For, if $\alpha\in T$ then $\sigma_T\sigma_{v(S)}(\alpha)=\alpha>0$, while if $\alpha\in \Delta\setminus T$, then $\sigma_{v(S)}(\alpha)$ is positive and $\sigma_T\sigma_{v(S)}(\alpha)$ is negative if and only if $\sigma_{v(S)}(\alpha)\in T$ which is absurd because then $\alpha=\sigma_{v(S)}^2(\alpha)<0$.
This also implies $T=-v(S)$.

We showed that if $(v,S)$ admits no complex descent then $-v(S)\subseteq \Delta$.
But then using proposition \ref{maxim} we get that every root in $S$ is maximal in $\Phi^+(v)$, so they must be incomparable.
By Property \ref{unic}, $S$ is the only incomparable set of maximum rank and $v$ is uniquely identified by the fact that every element in $S$ is maximal in $\Phi^+(v)$.
It follows that there is a unique minimal orbit $(v_0,S_0)$.

It is left to prove that $v$ is minimal among the $w\in W^P$ for which there is $T\subseteq \Phi^+(v)$ of maximum rank.
Consider $(w,T)\in\RM$.
Then $(v_0,S_0)\leq (w,T)$ and by the characterization of Theorem \ref{ordineGM} it must be $v_0\leq w$.

\item Fix $(v,S)\in\RM$.
Then we know by Lemma \ref{listofaction} that if $\alpha\in\Delta$ is a complex descent, $P_\alpha vx_S$ contains only two orbits
(we are in case $b2$).
The smaller one, let's call it $(u,R)$, is still in $\RM$ because either $R=S$ or $R=s_{v^{-1}(\alpha)}(S)$.
\item By point $(3)$ every non minimal orbit admits a complex descent and the smaller orbit in this descent is still of maximum rank.
Inductively we obtain the claim.
\item we saw above that both $(v,S)$ and $(u,T)$ admit a chain of descents to $(v_0,S_0)$.
But the descents are determined by the associated involutions, so they admit the same chain of descent, from which follows that they are the same orbit.
\item In the first case $\beta$ can't be orthogonal to $S$, so the ascent must be complex.
If $-\beta\in\Psi$, instead, $\alpha$ is a descent and if $-\beta\notin S$ we know that $Pvx_S=Bvx_S\sqcup Bs_\alpha vx_S$.
\end{enumerate}
\end{proof}

\begin{lemma}
Let $(v,S)\in\RM$.
Then $(v,S_0)\leq (v,S)$.
\end{lemma}
\begin{proof}
Take a sequence $(\alpha_1,\ldots,\alpha_n)$ for $(v,S)$ like the one in point \ref{seq} of Lemma \ref{prop}.
Suppose $(\alpha_{i_1},\ldots,\alpha_{i_k})$ is the subsequence that contains exactly the ascents on $W^P$.
Then if $(v',S_0)=M(\alpha_{i_1},\ldots,\alpha_{i_k}).(v_0,S_0)$ we have
\[
\begin{array}{ccc}
(v',S_0)\leq (v,S) & \text{and} & v=v'
\end{array}
\]
\end{proof}

\begin{cor}\label{soprasotto}
Let $S\in\RM$ with $S\neq S_0$.
Then for every $\beta\in S_0$ there is $\alpha\in S$ such that $\alpha<\beta$.
\end{cor}
\begin{proof}
Take $v\in W^P$ such that $S\subseteq \Phi^+(v)$.
Then $(v,S_0)<(v,S)$ if and only if $B_ve_{S_0}\subseteq \cdr{B_ve_S}$.
Recall that $\mathfrak{p}^u=\bigoplus_{\gamma\in\Psi} \mathfrak{u}_\gamma$.
Define $\Psi'=\left\lbrace \gamma\in\Psi\mid \exists \beta\in S \textit{ such that }\beta\leq \gamma\right\rbrace$ and $V=\bigoplus_{\gamma\in \Psi'} \mathfrak{u}_\gamma\subseteq \mathfrak{p}^u$.
Then 
\[B_ve_S\subseteq V\]
The claim follows by noting that $V$ is closed in $\mathfrak{p}^u$, so $B_ve_{S_0}\subseteq V$.
\end{proof}

The following theorem is the central result for characterizing the Bruhat $\mathcal{G}$-order in the simply laced case.
\begin{theorem}\label{ordmax}

Let $(u,T)$ and $(v,S)$ be admissible pairs.
Suppose that both $T$ and $S$ are maximal orthogonal subsets of $\Psi$ and that there is the following relation between the associated involutions
\[ \sigma_{u(T)}<\sigma_{v(S)}\]
Suppose, at last that
\[\sigma_{v(S)}=s_{\alpha_1}\cdots s_{\alpha_n}\sigma_{v_0(S_0)}s_{\alpha_n}\cdots s_{\alpha_1}\]
Then there is  a sequence $1\leq i_1<\ldots<i_k\leq n$ such that
\[\sigma_{v(S)}=s_{\alpha_{i_1}}\cdots s_{\alpha_ {i_n}}\sigma_{v_0(S_0)}s_{\alpha_{i_n}}\cdots s_{\alpha_{i_1}}\]
\end{theorem}

\begin{oss}
Note that this implies that the Bruhat order on the maximum rank orbit and the Bruhat $\mathcal{G}$-order are the same.
\end{oss}

The proof of Theorem \ref{ordmax} will be divided in four steps.
Set once and for all $d=L(\sigma_{v_0(S_0)})$ and put $\sigma_{v_0(S_0)}=s_{\gamma_1}\ldots s_{\gamma_d}$.
Put $\Psi_0=-v_0(S_0)$.
\begin{lemma}\label{codim1}
In the hypothesis of Theorem \ref{ordmax} suppose that $L(\sigma_{v(S)})=L(\sigma_{u(T)})+1$.
Then Theorem \ref{ordmax} holds.
\end{lemma}
\begin{proof}
We will prove the lemma by induction on $L(\sigma_{v(S)})$.
If $L(\sigma_{v(S)})=d+1$, then $(u,T)=(v_0,S_0)$, so we know the theorem holds.

Now suppose $n>d+1$, $L(\sigma_{v(S)})=n$ and we are given the decomposition 
\[\sigma_{v(S)}=s_\alpha s_{\alpha_2}\cdots s_{\alpha_n}\sigma_{v_0(S_0)}s_{\alpha_n}\cdots s_{\alpha_1}s_\alpha\]
We want to show that 
\[\sigma_{u(T)}=s_\alpha s_{\alpha_2}\cdots \widehat{s_{\alpha_i}}\cdots s_{\alpha_n}\sigma_{v_0(S_0)}s_{\alpha_n}\cdots \widehat{s_{\alpha_i}}\cdots s_{\alpha_1}s_\alpha\]
That is, that we can obtain $\sigma_{u(T)}$ from $\sigma_{v(S)}$ by omitting one of the lateral pairs.
If $\alpha$ is an ascent for $\sigma_{u(T)}$ then $s_\alpha\circ\sigma_{u(T)}=\sigma_{v(S)}$ and we have the thesis.
If $\alpha$ is a complex descent for $\sigma_{u(T)}$ then we know by inductive hypothesis that $s_\alpha\circ\sigma_{u(T)}=s_{\alpha_1}\cdots \widehat{s_{\alpha_i}}\cdots s_{\alpha_ {n}}\sigma_{v_0(S_0)}s_{\alpha_{i_n}}\cdots\widehat{s_{\alpha_i}}\cdots s_{\alpha_{i_1}}$ which implies that $\sigma_{u(T)}=s_\alpha s_{\alpha_1}\cdots \widehat{s_{\alpha_i}}\cdots s_{\alpha_ {n}}\sigma_{v_0(S_0)}s_{\alpha_{i_n}}\cdots\widehat{s_{\alpha_i}}\cdots s_{\alpha_{i_1}} s_\alpha$ as wanted.

So suppose that $\alpha$ is a real descent for $\sigma_{u(T)}$.
Denote with $(u',T')$ one of the orbits in $\mathcal{E}_\alpha(u,T)$ and with $(v',S')$ the orbit in $\mathcal{E}_\alpha(v,S)$.
We know that $\sigma_{u'(T')}=s_\alpha\circ\sigma_{u(T)}<s_\alpha\circ \sigma_{v(S)}=\sigma_{v'(S')}$ and the difference in length (as involutions) is still 1.
Recall that for a generic involution $\sigma_R$ with $R$ strongly orthogonal we have $L(\sigma_R)=\frac{l(\sigma_R)+\#R}{2}$.
So
\[L\left(\sigma_{u'(T')}\right)=\frac{ l\left(\sigma_{u'(T')}\right)+\#T-1}{2}=L\left(\sigma_{v'(S')}\right)-1=\frac{l\left(\sigma_{v'(S')}\right)+\#S}{2}-1\]
Because $T'=T\setminus \left\lbrace \beta\right\rbrace$ for some root $\beta\in T$ while $\#S'=\#S$.
Hence, $l(\sigma_{v'(S')})=l(\sigma_{u'(T')})+1$ and $\sigma_{u'(T')}\leq \sigma_{(v'(S')}$.
This implies that $\sigma_{u'(T')}$ is obtained by cancelling a single simple reflection in a reduced expression of $\sigma_{v'(S')}$
This is possible if and only if 
\[\sigma_{u'(T')}=s_\alpha\circ\sigma_{u(T)}=s_{\alpha_{1}}\cdots s_{\alpha_ {n}}s_{\gamma_1}\ldots \widehat{s_{\gamma_i}}\ldots s_{\gamma_d}s_{\alpha_{n}}\cdots s_{\alpha_{1}}\]
To see this, suppose that 
\[s_\alpha\circ\sigma_{u(T)}=s_{\alpha_{i_1}}\cdots \widehat{s_{\alpha_{i_j}}}\cdots s_{\alpha_ {i_n}}s_{\gamma_1}\ldots s_{\gamma_d}s_{\alpha_{i_n}}\cdots s_{\alpha_{i_1}}\]
it is easy to see that if $\sigma$ is an involution then $v\sigma v^{-1}$ is again an involution for every $v\in W$.
It follows that
\[s_{\alpha_{i_{j+1}}}\cdots s_{\alpha_ {i_n}}s_{\gamma_1}\ldots s_{\gamma_d}s_{\alpha_{i_n}}\cdots s_{\alpha_{i_j}}\]
must be an involution and we know that
\[\tau=s_{\alpha_{i_{j+1}}}\cdots s_{\alpha_ {i_n}}s_{\gamma_1}\ldots s_{\gamma_d}s_{\alpha_{i_n}}\cdots s_{\alpha_{i_{j+1}}}\]
is an involution.
But this implies that $\tau(\alpha_{i_j})=\pm \alpha_{i_j}$ which is absurd because $\alpha_{i_j}$ should be a complex ascent for $\tau$.
We then have the following graph where we put $v=s_{\alpha_2}\cdots s_{\alpha_ {n}}$

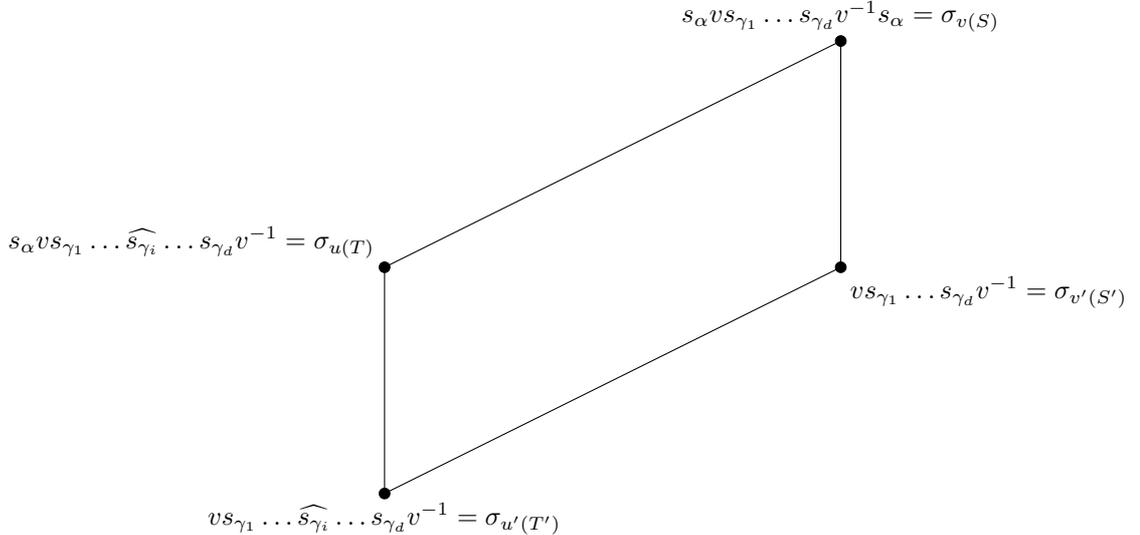
\begin{tikzpicture}
\filldraw[black] (8,0) circle (2pt) node[anchor=south] {$s_\alpha v s_{\gamma_1}\ldots s_{\gamma_d}v^{-1}s_\alpha=\sigma_{v(S)}$};
\filldraw[black] (8,-3) circle (2pt) node[anchor=north west]{$vs_{\gamma_1}\ldots s_{\gamma_d}v^{-1}=\sigma_{v'(S')}$};
\filldraw[black] (2,-3) circle (2pt) node[anchor=south east]{$s_\alpha vs_{\gamma_1}\ldots \widehat{s_{\gamma_i}}\ldots s_{\gamma_d}v^{-1}=\sigma_{u(T)}$};
\filldraw[black] (2,-6) circle (2pt) node[anchor=north]{$vs_{\gamma_1}\ldots \widehat{s_{\gamma_i}}\ldots s_{\gamma_d}v^{-1}=\sigma_{u'(T')}$};
\draw (2,-3) -- (8,0);
\draw (2,-6) -- (8,-3);
\draw (8,0) -- (8,-3);
\draw (2,-3) -- (2,-6);
\end{tikzpicture}

But if we look at the left part we see that $\alpha$ should be orthogonal to $v(\gamma_1),\ldots,\widehat{v(\gamma_i)},\ldots,v(\gamma_d)$.
On the other hand, the right part tells us that $\alpha$ is not orthogonal to $v(\gamma_1),\ldots,v(\gamma_d)$, which by Lemma \ref{due} would imply $-\alpha=v(\gamma_i)$ and that's absurd because $\sigma_{v'(S')}(\alpha)>0$.
\end{proof}

Recall that $S_0$ is the $\Psi$-part of the smallest orbit in $\RM$.

\begin{lemma}\label{above}
Let $\sigma_{u(S_0)}<\sigma_{v(S)}$ with $(u,S_0),(v,S)\in \RM$.
Write 
\[\sigma_{v(S)}=s_{\alpha_1}\cdots s_{\alpha_n}\sigma_{v_0(S_0)}s_{\alpha_n}\cdots s_{\alpha_1}\]
Then there is $1\leq i_1<\ldots<i_k\leq n$ such that
\[\sigma_{u(S_0)}=s_{\alpha_{i_1}}\cdots s_{\alpha_{i_k}}\sigma_{v_0(S_0)}s_{\alpha_{i_k}}\cdots s_{\alpha_{i_1}}\]
\end{lemma}
\begin{proof}
This is again an induction on $L(\sigma_{v(S)})$.
If $L(\sigma_{v(S)})=d$ it is clear.
Suppose $L(\sigma_{v(S)})>d$ and consider the complex descent $\alpha=\alpha_1$.
If $\alpha$ is an ascent for $\sigma_{u(S_0)}$ then we obtain the claim by applying the inductive hypothesis.
Similarly if it is a complex descent; we just need to observe that the $\Psi$-part of the descent is still $S_0$ by Corollary \ref{soprasotto}.

Suppose then that $\alpha$ is a real descent for $(u,S_0)$.
This means that $-u^{-1}(\alpha)\in S$ and it is maximal in $\Phi^+(u)$.
We claim that this implies $\sigma_{u(S_0)}\leq s_\alpha\sigma_{u(S)}s_\alpha$, so we can still obtain the thesis by applying the inductive hypothesis.
Recall that we can write $\sigma_{v_0(S_0)}=s_{\gamma_1}\cdots s_{\gamma_r}$ with $\gamma_i$ pairwise orthogonal simple roots.
Then, by the hypothesis and the characterization of the Bruhat order in $W$ there are two subsequences $(h_1,\ldots, h_l),(k_1,\ldots, k_t)\subseteq (1,2,\ldots,n)$ and $p<r$ such that
\[\sigma_{u(S_0)}=s_\alpha s_{\alpha_{k_1}}\cdots s_{\alpha_{k_t}}s_{\gamma_1}\ldots s_{\gamma_p}s_{\alpha_{h_l}}\ldots s_{\alpha_{h_1}}\]
Basically, we wrote $\sigma_{u(S_0)}$ as a sub-word of $\sigma_{v(S)}$ by cancelling some of the central $s_{\gamma_i}$, some of the $s_{\alpha_i}$ on both sides and some of the $s_{\gamma_i}$ only on one side.
Note that we can suppose that we cancelled $s_\alpha$ on the right side because we supposed that $\alpha$ is a real descent and if we could cancel it on both side we would instantly have the claim.

We will now show that $\alpha\in \Psi_0$ that is, that $s_\alpha$ coincides with one of the $s_{\gamma_i}$ we cancelled from the center and that 
\[\sigma_{u(S_0)}=s_{\alpha_{k_1}}\cdots s_{\alpha_{k_t}}s_\alpha s_{\gamma_1}\ldots s_{\gamma_p}s_{\alpha_{h_l}}\ldots s_{\alpha_{h_1}}\]
Note that this would imply $\sigma_{u(S_0)}\leq s_\alpha\circ\sigma_{v(S)}$ and, by induction, the thesis.

For the first part, note that if $w,w'\in W^P$ and $\beta$ is maximal both in $\Phi^+(w)$ and in $\Phi^+(w')$, then $w(\beta)=w'(\beta)$.
To see this, let $w_0$ be the unique element in $W^P$ such that $\Phi^+(w_0)=\Phi^+(w)\cap \Phi^+(w')$.
Then $\beta$ is maximal in $\Phi^+(w_0)$, so we only need to show the claim for $w'=s_\tau w$.
But $(s_\tau w)(\beta)=s_\tau(w(\beta))=w(\beta)$ because $w^{-1}(\tau)$ is orthogonal to $\beta$ from which follows that $\tau$ is orthogonal to $w(\beta)$.
Note that this implies that whenever $u\geq v_0$ every real descent for $(u,S_0)$ must be in $\Psi_0$.

For the second part, suppose that $\alpha_{k_1},\ldots, \alpha_{k_r}\in \Psi_0$ and $\alpha_{k_{r+1}}\notin \Psi_0$.
Then $\alpha,\alpha_{k_1},\ldots, \alpha_{k_r}$ must be all different and hence they commute.
Put $\sigma=s_{\alpha}s_{\alpha_{k_1}}\cdots s_{\alpha_{k_r}}$.
Note that $\beta=\alpha_{k_{r+1}}$ must be a complex descent for $\sigma'=\sigma \sigma_{u(S_0)}$ because $\sigma'$ is the involution of an orbit of the form $(u',S')$ with $S\subseteq S_0$.
It follows that $\sigma'=s_{\beta}\sigma''s_{\beta}$.
But this implies that $\beta$ is also a complex descent for $\sigma_{u(S_0)}$ and by Corollary \ref{soprasotto} it must be a descent on $u$.
Hence $s_\beta u<u$ and $-u^{-1}(\beta)\in \Phi^+(u)$ is maximal.
But by hypothesis we also had that $-u^{-1}(\alpha)$ as well as $-u^{-1}(\alpha_{k_i}$ for $i\leq r$ was maximal in $\Phi^+(u)$ which then implies that $-u^{-1}(\alpha)$ and $-u^{-1}(\beta)$ are orthogonal, hence also $\alpha$ and $\beta$ are orthogonal.
Inductively we conclude.

\end{proof}

We will now show that if we restrict to orbits of maximum rank the characterization of Theorem \ref{ordineGM} can be simplified.

\begin{cor}\label{ordinv}
Let $\sigma_{u(T)}\leq\sigma_{v(S)}$ with $(u,T),(v,S)\in\RM$.
Then $(u,T)\leq (v,S)$.
\end{cor}
\begin{proof}
Recall that $(u,T)\leq (v,S)$ if and only if $\sigma_{u(T)}\leq \sigma_{v(S)}$ and $[v\sigma_S]^P\leq [u\sigma_T]^P\leq u\leq v$, where $[w]^P$ is the representative in $W^P$ of the class $wW_P$.

We have the first inequality by hypothesis.
Now consider $[v\sigma_S]^P$.
By point \ref{seq} of Lemma \ref{prop}, we can write 
\[\sigma_{v(S)}=s_{\alpha_1}\cdots s_{\alpha_n} v_0\sigma_{S_0}v_0^{-1}s_{\alpha_n}\cdots s_{\alpha_1}\]
We can also write this in another more useful way.
Let $(\alpha_n,\ldots,\alpha_1)$ be the sequence of complex ascents that links $(v_0,S_0)$ to $(v,S)$ and $(v_1,S_1)<\cdots<(v_n,S_n)=(v,S)$ the intermediate steps.
The first ascent $\alpha_n$ must necessarily be an ascent on $v_0$.
Suppose that $\alpha_i$ is the first root that is an ascent on $\Psi$ for $(v_{n-i},S_{n-i})$.
Then $v_{n-i}=s_{\alpha_{i+1}}\cdots s_{\alpha_n}v_0$ and $v_{n-i}^{-1}(\alpha_i)=\beta\in\Delta$ so we can write
\[\sigma_{v_{n-i+1}(S_{n-i+1})}=\overbrace{s_{\alpha_{i+1}}\cdots s_{\alpha_n}v_0}^{v_{n-i+1}}\underbrace{s_\beta \sigma_{S_0} s_\beta}_{\sigma_{S_{n-i+1}}} v_0^{-1}s_{\alpha_n}\cdots s_{\alpha_{i+1}}\]
inductively, if we split the sequence $(\alpha_n,\ldots,\alpha_1)$ in two subsequences $(\tau_l,\ldots,\tau_1),(\beta_m,\ldots,\beta_1)$ such that the $\tau$ are the ascents on $W^P$ and the $\beta$ are the  inverse images of the ascents on $\Psi$ we get
\[\sigma_{v(S)}=\overbrace{s_{\tau_1}\cdots s_{\tau_n}v_0}^v \underbrace{s_{\beta_1}\cdots s_{\beta_n}\sigma_{S_0} s_{\beta_n}\cdots s_{\beta_1}}_{\sigma_S}v_0^{-1}s_{\tau_n}\cdots s_{\tau_1}\]
It follows that
\[ [v\sigma_S]^P=[vs_{\beta_1}\cdots s_{\beta_n}\sigma_{S_0} s_{\beta_n}\cdots s_{\beta_1}]^P=[vs_{\beta_1}\cdots s_{\beta_n}\sigma_{S_0}]^P\]
where the last step is true because $\beta_i\in\Delta_P$.
Now $s_\beta\sigma_{S_0}=\sigma_Ss_{\sigma_{S_0}(\beta)}$ and if $\sigma_{S_0}=s_{\gamma_1}\cdots s_{\gamma_d}$ we have without loss of generality and up to sign
\[\sigma_S(\beta)=\left\lbrace\begin{array}{l}
								\pm\beta \\
								\pm\beta+\gamma_1+\gamma_2 \\
								\pm\beta+\gamma_1-\gamma_2
								\end{array} \right.\]
We observe that the middle case is impossible if $\beta\in \Delta_P$ because the coefficient of $\alpha_P$ of $\pm\beta+\delta_1+\delta_2$ should be $2$, which it can't be.
Then in both the remaining cases $\sigma_{S_0}(\beta)\in \Phi_P$, so $s_{\sigma_{S_0}(\beta)}\in W_P$.
This implies inductively that $[v\sigma_S]^P=[v\sigma_{S_0}]^P$.
That is, $\left[v\sigma_S\right]^P$ doesn't depend on $S$ whenever $S$ is of maximum rank.

To conclude, observe that $(u,S_0)<(u,T)$, hence $\sigma_{u(S_0)}\leq \sigma_{u(T)}\leq \sigma_{v(S)}$ and it is enough to show $(u,S_0)\leq (v,S)$.
But Lemma \ref{above} tells us that there is a sequence of ascents $(\alpha_1,\ldots,\alpha_n)$ between $(v_0,S_0)$ and $(v,S)$ with a subsequence $\left(\alpha_{i_1},\ldots,\alpha_{i_r}\right)$ that is itself a sequence of ascents between $(v_0,S_0)$ and $(u,S_0)$.
This implies $(u,S_0)\leq (v,S)$ by Lemma \ref{storder}.

\end{proof}

We still need to prove the theorem in another particular case.
To do this, we will now introduce the concept of generalized admissible pairs.
\begin{defin}
A \textit{generalized admissible pair} is a pair $[v,S]$ where $v\in W^P$ and $S\in \Psi$ is orthogonal.
We will denote them with square bracket.
\end{defin}
There is an obvious surjective map
\begin{align*}
\varphi:\text{Generalized admissible pairs}&\longrightarrow \text{Admissible pairs}\\
[v,S]\mapsto (v,S\cap\Phi^+(v))
\end{align*}
moreover, if we restrict to $\varphi|_{\varphi^{-1}(\RM)}\colon\varphi^{-1}(\RM)\longrightarrow \RM$ it is a bijection.
Now consider only the generalized admissible pairs where the $\Psi$-part is of maximum rank.
The generalized admissible pairs admit an action of $\Delta$ which derives from the action on the admissible pairs and on the related orbit.
Fix $\alpha\in \Delta$, take $[v,S]$ and consider $(v,T)=\varphi([v,S])$.
Then the action of the minimal parabolic $P_\alpha$ gives an action $m_\alpha(v,T)=(v',T')$.
Then we put $\alpha.[v,S]=[v',S']$ where $S'=S$ unless $T'=s_\beta(T)$, in which case $S'=s_\beta(S)$.
Note that this commutes with $\varphi$.

\begin{lemma}
Suppose $\sigma_{v(T)}<\sigma_{v(S)}$ (with the same $W$-part).
Then Theorem \ref{ordmax} holds.
\end{lemma}
\begin{proof}
As before, we work by induction on $L(\sigma_{v(S)})$.
If $L(\sigma_{v(S)})=d$ then there is nothing to show.

Now suppose $L(\sigma_{v(S)})=n>d$ and suppose given a decomposition 
\[\sigma_{v(S)}=s_\alpha s_{\alpha_2}\cdots s_{\alpha_n}\sigma_{v_0(S_0)}s_{\alpha_n}\cdots s_{\alpha_2}s_\alpha\]
We know that $\alpha$ is a complex descent for $(v,S)$.
If $\alpha$ is an ascent or a complex descent for $(v,T)$ we proceed by induction as before.
Suppose then that $\alpha$ is a real descent for $(v,T)$.
Now, beginning from $(v,T)<(v,S)$, we can apply the sequence of descents $(\alpha=\alpha_1,\alpha_2,\ldots,\alpha_n)$ to both sides and we obtain a sequence of pairs for which still $(v_i,T_i)<(v_i,S_i)$.
Note that $T_1$ has one less element then $T$ so we can never have quality.
Moreover, in this sequence the $W$-part is always the same and decreasing.
The last step is $(v_0,T_0)<(v_0,S_0)$ where $(v_0,S_0)$ is the minimal orbit of maximum rank.
This implies that $\sigma_{v_0(T_0)}=s_{\gamma_1}\cdots s_{\gamma_r}$ where $\gamma_1,\ldots,\gamma_r\in\Psi_0$ are pairwise orthogonal simple roots.
On the other hand, we can do the same with generalized admissible pairs; in the last step we obtain a situation as below
\begin{center}
\begin{tikzpicture}
\node at (2,0) [anchor=south]{$[v_0,T']$};
\node at (5,0) [anchor=south]{$[v_0,S_0]$};
\node at (2,-2) [anchor=north]{$(v_0,T_0)$};
\node at (5,-2) [anchor=north]{$(v_0,S_0)$};
\node at (3.5,-2.1) [anchor=north]{$<$};
\draw[->] (2,0) -- node[left]{$\varphi$} (2,-2);
\draw[->] (5,0) -- node[right]{$\varphi$}(5,-2);
\end{tikzpicture}
\end{center}
This means that $T'\cap \Phi^+(v_0)=T_0$ and given that $T_0$ has less element than $T$ and consequently less element then $S_0$ it must be $T_0\neq S_0$, hence $T'\neq S_0$.
By Corollary \ref{soprasotto} this means that there is $\alpha\in T'$ that is smaller than at least an element of $S_0$.
Then $v_0(\alpha)<0$, so $\alpha\in T_0$ and $-v_0(\alpha)$ is not a simple root.
We would then have
\[-v_0(\alpha)=\sigma_{v_0(T_0)}(v_0(\alpha))=s_{\gamma_1}\cdots s_{\gamma_r}(v_0(\alpha))\]
which implies $v(\alpha)=-\gamma_i$ which is a contradiction.
\end{proof}

At last, we can prove Theorem \ref{ordmax}.
We will repeat here the statement.
Recall that there is a maximum element in $W^P$ which we will call $\omega_P$.
\begin{manualtheorem}{\ref{ordmax}}

Let $(u,T)$ and $(v,S)$ be admissible pairs.
Suppose that both $T$ and $S$ are maximal orthogonal subsets of $\Psi$ and that there is the following relation between the associated involutions
\[ \sigma_{u(T)}<\sigma_{v(S)}\]
Suppose, at last that
\[\sigma_{v(S)}=s_{\alpha_1}\cdots s_{\alpha_n}\sigma_{v_0(S_0)}s_{\alpha_n}\cdots s_{\alpha_1}\]
Then there is  a sequence $1\leq i_1<\ldots<i_k\leq n$ such that
\[\sigma_{u(T)}=s_{\alpha_{i_1}}\cdots s_{\alpha_ {i_r}}\sigma_{v_0(S_0)}s_{\alpha_{i_r}}\cdots s_{\alpha_{i_1}}\]
\end{manualtheorem}
\begin{proof}
Note that by Corollary \ref{ordinv} we know $(u,T)\leq (v,S)$.

We work by downward induction on $l(u)$ starting from the maximum length, which is $l(\omega_P)=\#\Psi$.
Fix a decomposition $\sigma_{v(S)}=s_{\alpha_1}\cdots s_{\alpha_n}\sigma_{v_0(S_0)}s_{\alpha_n}\cdots s_{\alpha_1}$.
If $u=\omega_P$, then by Corollary \ref{ordinv} and the characterization of the Bruhat order of the orbits we would have also $v=\omega_P$, so the claim follows by the previous lemma.
Now suppose $l(u)=k<\#\Psi$.
Then $u<\omega_P$ and there must be $\alpha\in\Delta$ such that $u<s_\alpha u\leq\omega_P$.
In particular $\alpha$ is a complex ascent (on $u$) for $(u,T)$.
Note that there is $(u',T)=m_\alpha(u,T)$ such that $\sigma_{u'(T)}=s_\alpha\circ\sigma_{u(T)}>\sigma_{u(T)}$.

Now suppose that $\alpha$ is a descent for $\sigma_{v(S)}$.
Then $\sigma_{u(T)}<s_\alpha\circ\sigma_{u(T)}\leq\sigma_{v(S)}$.
By inductive hypothesis we know that $s_\alpha\circ\sigma_{u(T)}=s_{\alpha_{i_1}}\cdots s_{\alpha_{i_{r+1}}}\sigma_{v_0(S_0)}s_{\alpha_{i_{r+1}}}\cdots  s_{\alpha_{i_1}}$ and by Lemma \ref{codim1} we know that the claim also holds for $\sigma_{u(T)}$.

Suppose now that $\alpha$ is an ascent for $\sigma_{v(S)}$ and $P_\alpha vx_S=Bvx_S$.
Then $(u',T')\leq (v,S)$, hence $\sigma_{u'(T')}\leq \sigma_{v(S)}$ and we conclude as above.

If instead $\alpha$ is an ascent for $\sigma_{v(S)}$ which is also an ascent for $(v,S)$, that is if $P_\alpha vx_S\neq Bvx_S$ we would have 

\[s_{\alpha}\circ\sigma_{u(T)}<s_\alpha\circ\sigma_{v(S)}\]
to which we can apply the inductive hypothesis with

\[s_\alpha\circ\sigma_{v(S)}=s_\alpha s_{\alpha_1}\cdots s_{\alpha_n}\sigma_{v_0(S_0)}s_{\alpha_n}\cdots s_{\alpha_1} s_\alpha\]
Then if the decomposition we obtain for $s_\alpha\circ\sigma_{u(T)}$ doesn't have $s_\alpha$ at the extremities we would have $s_\alpha\circ\sigma_{u(T)}<\sigma_{v(S)}$ so we can use the same reasoning as before.
If instead it has $s_\alpha$ at the extremities, then it is clear that cancelling those $s_\alpha$ gives the decomposition for $\sigma_{u(T)}$ we were looking for.

\end{proof}

We will now see how the preceding result let us study the local systems on the $B$-orbits on $G/L$.
The first thing is to analyse the possible ascents and descents when the root system is simply laced.
\begin{lemma}\label{c1no}
Suppose that $\Phi$ is simply laced.
Then case $c1)$ from Lemma \ref{listofaction} never happens.
\end{lemma}
\begin{proof}
Suppose that $(v,S)$ is admissible and that there is $\alpha\in\Delta$ such that $\sigma_{v(S)}(\alpha)=\alpha$.
Note that, $\Phi$ being simply laced, $\gamma,\delta\in\Phi$ can be added or subtracted if and only if $(\gamma,\delta)\neq 0$.
Then $\sigma_{v(S)}(\alpha)=\alpha$ implies that $\alpha$ cannot be added or subtracted to any element of $v(S)$.
This means that if $\beta=v^{-1}(\alpha)\in \Delta_P$, then $P_\alpha vx_S=Bvx_S$, while if $\beta\in\Psi$, then $P_\alpha v x_S$ contains three orbits
\end{proof}

The above result is very important because case $c1)$ is the only case in which a local system admits more than one extension.
It follows that, if the root system is simply laced, the trivial local system never extends to a non-trivial local system, so the subset of trivial local systems must be a connected component of the Hasse diagram.
We will now study which orbit can have non-trivial local systems.

\begin{prop}
Let $G$ be a simply connected linear algebraic group of type $\bf{ADE}$.
Then, if $\Psi$ doesn't verify property \ref{unic}, no orbit admits non-trivial root systems.
If instead $\Psi$ verifies property \ref{unic}, then a $B$-orbit $(v,S)$ on $G/L$ admits a non-trivial local system if and only if $S$ is of maximum rank.
In this case, it admits a unique non-trivial local system.
\end{prop}
\begin{proof}
We know that the local systems on $Bvx_S$ are in correspondence with the connected components of the stabilizer in $B$ of $vx_S$.
To compute this, it is enough to study the stabilizer in $T$ of $vx_S$ given that the unipotent part is connected.
So, the proof of this theorem is just computing the connected components of $\stab_T(vx_S)$ for varying $(v,S)$.
The first to note is that $\stab_T(vx_S)=v\stab_T(x_S)v^{-1}$.
Suppose $\Phi$ is of type $\bf{A}$, so we can assume $G=\matr{SL}(n,\mathbb{C})$.
Then we can suppose $\Phi=\left\lbrace \pm(e_i-e_j)\mid i<j\right\rbrace\subseteq \langle e_1,\ldots,e_n\rangle_{\mathbb{R}}\cong \mathbb{R}^n$ and $\Delta=\left\lbrace e_i-e_{i+1}\right\rbrace_{1\leq i\leq n-1}$.
Then for every $\alpha_h=e_h-e_{h+1}\in\Delta$ the parabolic subgroup associated to $\Delta\setminus\left\lbrace\alpha_h\right\rbrace$ has abelian unipotent radical.
In this case $\Psi=\left\lbrace e_i-e_j\mid i\leq h<h+1\leq j\right\rbrace$.
Obviously $T=\left\lbrace \diag(a_1,\ldots,a_n)\right\rbrace$ and $\left(e_i-e_j\right)\left(\diag(a_1,\ldots,a_n)\right)=a_ia_j^{-1}$.
Moreover there is the additional condition that $\prod_{i=1,\ldots,n} a_i=1$.
It follows that every root $e_i-e_j$ in $S$ gives a condition of the type $a_i=a_j$.
If there is $h_0\in\left\lbrace 1,\ldots,n\right\rbrace$ such that $e_{h_0}$ is orthogonal to every root in $S$, then the additional condition $\prod_{i=1,\ldots,n} a_i=1$ can be written as 
\[a_{h_0}=\prod_{i=\left\lbrace 1,\ldots, \widehat{h_0},\ldots,n\right\rbrace} a_i^{-1}\]
and it is clear that $\stab_T(vx_S)\cong \mathbb{C}^{n-1-\#S}$.
Suppose that there is no $e_i$ orthogonal to every root in $S$.
Then it must be $n$ even and $h=\frac{n}{2}$ (for $G=\matr{SL}_n$ this is equivalent to $\Psi$ having Property \ref{unic}).
Moreover in this case there is no $e_i$ orthogonal to every root in $S$ if and only if $S$ is of maximum rank.
If we denote $I=\left\lbrace 1\leq i\leq n\mid e_i-e_{k(i)}\in S\text{ for some }i<k(i)\right\rbrace$, then the stabilizer is defined by
\[
\left\lbrace 
\begin{array}{lr}
a_{k(i)}=a_i & i\in I\\
\prod_{i\in I} a_i^2=1 & 
\end{array}
\right.
\]
And this is two copies of $\mathbb{C}^{\frac{n}{2}-1}$.

The $\bf{D}$ and $\bf{E}$ cases are similar.
\end{proof}

We can now describe the Hasse diagram of the simply laced root systems.
\begin{theorem}\label{risADE}
Suppose that the linear algebraic group $G$ is simply connected and the root system $\Phi$ is simply laced.
If $\Psi$ doesn't verify Property \ref{unic}, then all local systems are trivial and $\mathcal{D}=\mathcal{D}_0$.

If instead $\Psi$ verifies Property \ref{unic}, then:
\begin{enumerate}
\item the orbits of maximum rank admit exactly two non-isomorphic local systems, one being trivial and one being non-trivial.
The other orbits admit only the trivial local system;
\item the subset of all the orbits with trivial root system is a connected component of the Hasse diagram, while the subset of the orbits of maximum rank with non-trivial root system is another connected component;
\item in every connected component, the order between the elements is the Bruhat order between the underlying orbits.
\end{enumerate}
\end{theorem}
\begin{proof}
\begin{enumerate}
\item We proved this in the proposition above;
\item by Proposition \ref{ordtriv} the order among the orbits with trivial local systems coincides with the Bruhat order on the respective orbits, so the set of all the orbits with trivial local systems is connected.
Similarly, the subset of orbits of maximum rank with non-trivial local system is connected because there are ascents between the minimum orbit $(v_0,S_0)$ and any $(v,S)\in\RM$.
We now need to show that these are in fact the connected components of the Hasse diagram, that is, that they are disconnected.
To do this it is enough to show that if $\gamma\in\mathcal{D}$ is trivial and $\tau\in\mathcal{D}$ is non-trivial then there is no $\alpha\in\Delta$ such that $\tau\in\alpha\circ\gamma$ or $\gamma\in\alpha\circ\tau$.
By Lemma \ref{c1no} $\alpha\circ\gamma$ can contain only trivial local systems, so $\tau\notin\alpha\circ\gamma$.
Then, suppose $\gamma\in\alpha\circ\tau$.
From this we know that $\gamma$ admits an extension to the orbit of $\tau$ and by Lemma \ref{listofaction} this extension is unique.
But, $\gamma$ being trivial, it admits at least the trivial extension, so we would have $\tau$ trivial which is a contradiction;

\item for the connected component with trivial local system we already know the claim is true (Proposition \ref{ordtriv}).
We also know that if $(Bvx_S,\gamma)\leq(Bux_R,\tau)$ in $\mathcal{D}$, then $Bvx_S\leq Bux_R$.
Now take $(v,S),(u,R)\in\RM$ and denote with $\gamma$ and $\tau$ the respective non-trivial local systems.
If $(v,S)<(u,R)$, then by Theorem \ref{ordmax} there is a chain of ascents $(\alpha_1,\ldots,\alpha_n)$ from $(v_0,S_0)$ to $(u,R)$ and a sub-chain $(\alpha_{i_1},\ldots,\alpha_{i_r})$ from $(v_0,S_0)$ to $(v,S)$.
Given that in every ascent the non-trivial local system extends to a non-trivial local system, if we apply the same ascents to $\left(Bv_0x_{S_0},\delta\right)$ with $\delta$ non-trivial, by Property \ref{propordlb} we get $(Bvx_S,\gamma)<(Bux_R,\tau)$.
\end{enumerate}
\end{proof}
\section{The type \bf{B} case}\label{B}

When $\Phi$ is of type $\bf{B}$ the situation is easy enough to analyse manually.
Think of $\bf{B}_n$ as the set
\[\left\lbrace \pm e_i\pm e_j\right\rbrace\cup\left\lbrace \pm e_i\right\rbrace\subseteq \langle e_1,\ldots,e_n\rangle_\mathbb{R}\]
We choose as a basis
\[\Delta=(\underbrace{e_1-e_2}_{\alpha_1},\ldots,\underbrace{e_{n-1}-e_n}_{\alpha_{n-1}},\underbrace{e_n}_{\alpha_n})\]
which gives $\Phi^+=\left\lbrace e_i-e_j\mid i<j\right\rbrace\cup \left\lbrace e_i\right\rbrace$.
We will suppose $n>2$

With this choices, the only parabolic subgroup with abelian unipotent radical is the one corresponding to $\Delta\setminus\left\lbrace e_1-e_2\right\rbrace$, so $\alpha_P=e_1-e_2$ and $\Psi=\left\lbrace e_1\pm e_i\right\rbrace_{i=2,\ldots,n}\cup\left\lbrace e_1\right\rbrace$.

As a first thing we want to compute which orbit admits a non-trivial local system.
Suppose $G$ simply connected, then $\Phi^\vee$ is generated by $\omega_i=\langle \alpha_i,\bullet\rangle$ and we know that $T=\Phi^\vee\otimes\mathbb{C}^*$.
Then every element in $T$ can be written as $\sum a_i\omega_i$ with $a_i\in\mathbb{C}$ and for every $\tau\in\Phi$ we get
\[\tau\left(\sum a_i\omega_i\right)=\prod a_i^{\omega_i(\tau)}\]

The orthogonal subset in $\Phi$ are either the singletons or sets of the form $\left\lbrace e_1-e_j,e_1+e_j\right\rbrace$ and we obtain
\begin{align*}
e_1(\sum a_i\omega_i)&=a_1 \\
\left(e_1-e_j\right)\left(\sum a_i\omega_i\right)&=a_1 a_{j-1}a_j^{-1} \\
\left(e_1+e_j\right)\left(\sum a_i\omega_i\right)&=a_1 a_{j-i}^{-1}a_j \\
\end{align*}

We want to compute $\stab_T(x_S)$.
If $S$ is a singleton, we only have an equation (either $a_1=1,a_1 a_{j-1}a_j^{-1}=1$ or $a_1 a_{j-i}^{-1}a_j=1$).
In these cases the stabilizer is isomorphic to $\left(\mathbb{C}^*\right)^{n-1}$ so it is connected.

If instead we have a set of the form $S=\left\lbrace e_1-e_j,e_1+e_j\right\rbrace$ the equations are

\begin{align*}
a_1 a_{j-1}a_j^{-1}&=1\\
a_1 a_{j-i}^{-1}a_j&=1
\end{align*}
with solutions 
\[\begin{array}{ccc}
\left\lbrace\begin{array}{l}
a_j=a_{j-1}\\
a_1=1
\end{array}\right.
& \text{ or } & \left\lbrace\begin{array}{l}
a_j=-a_{j-1}\\
a_1=-1
\end{array}\right.
\end{array}
\]
This means that the stabilizer is isomorphic to two copies of $\left(\mathbb{C}^*\right)^{n-2}$ and has two connected components.
These results are summarized in the following proposition.

\begin{prop}\label{propBlb}
Let $G$ be a simply connected algebraic group of type $\bf{B}$.
Then the $B$-orbits in $G/L$ which are represented by a pair $(v,S)$ with $\#S=2$ admit exactly two non-isomorphic local systems.
The others admit only the trivial local system.
\end{prop}

Note that the orbits $Bvx_S$ where $\#S=2$ are exactly the orbits of maximum rank in this case.

Now that we know how many non-isomorphic local systems there are on every orbit, we want to describe the Hasse diagram of $\mathcal{D}$.
First of all we need to understand $W^P$, which in this case is quite simple.
In fact, $\Psi$ is totally ordered because
\[e_1-e_2<e_1-e_3<\cdots<e_1-e_n<e_1<e_1+e_n<\cdots<e_1+e_2\]
so there is a correspondence between $\Psi$ and $W^P$.
Note that for every $w\in W$ we have $w(e_i)=\pm e_k$ because that is true for the simple reflections.
Moreover, if $w\in W^P$, then $w(\alpha)>0$ for every root in $\Delta_P$.
This means that $w(e_i)=e_{k(i)}$ for every $i\neq 1$ and $k(i)<k(j)$ whenever $i<j$.

With some calculation we get the following result:
\begin{prop}
Let $v=v_{\tau}$ be the element in $W^P$ such that $\tau$ is the maximum root in $\Phi^+(v)$.
Then:
\begin{enumerate}
\item if $\tau=e_1-e_j$ then
\[\left\lbrace \begin{array}{ll}
v(e_1)=e_j & \\
v(e_i)=e_{i-1} & \forall 1<i\leq j\\
v(e_i)=e_i & \forall j< i
\end{array}
\right.
\]
\item if $\tau=e_1$ then
\[\left\lbrace \begin{array}{ll}
v(e_1)=-e_n & \\
v(e_i)=e_{i-1} & \forall i\neq 1\\
\end{array}
\right.
\]
\item if $\tau=e_1+e_j$ then
\[\left\lbrace \begin{array}{ll}
v(e_1)=-e_{j-1} & \\
v(e_i)=e_{i-1} & \forall 1<i<j\\
v(e_i)=e_i & \forall j\leq i
\end{array}
\right.
\]
\end{enumerate}
\end{prop}

We know that if $(v,S)$ is an admissible pair and $\alpha\in \Delta$, then many information on the action of $P_\alpha$ are encoded in the root $\beta=v^{-1}(\alpha)$.
Thanks to the proposition above we can easily compute $v^{-1}(\alpha)$ for all $v\in W^P$.

\begin{prop}\label{azione WP}
Let $v=v_\tau$ as in the proposition above.
Then
\begin{enumerate}
\item if $\tau=e_1-e_j$ and $j\neq n$, then
\[
\left\lbrace \begin{array}{ll}
v^{-1}(e_i-e_{i+1})=e_{i+1}-e_{i+2} & \forall i\leq j-2\\
v^{-1}(e_{j-1}-e_{j})=-e_1+e_{j}\\
v^{-1}(e_{j}-e_{j+1})=e_1-e_{j+1}\\
v^{-1}(e_i-e_{i+1})=e_i-e_{i+1} & \forall i> j\\
v^{-1}(e_n)=e_n
\end{array}
\right.
\]
\item if $\tau=e_1-e_n$, then
\[
\left\lbrace \begin{array}{ll}
v^{-1}(e_i-e_{i+1})=e_{i+1}-e_{i+2} & \forall i\leq n-2\\
v^{-1}(e_{n-1}-e_{n})=-e_1+e_{n}\\
v^{-1}(e_n)=e_1
\end{array}
\right.
\]
\item if $\tau=e_1$, then
\[
\left\lbrace \begin{array}{ll}
v^{-1}(e_i-e_{i+1})=e_{i+1}-e_{i+2} & \forall i\leq n-2\\
v^{-1}(e_{n-1}-e_{n})=e_1+e_{n}\\
v^{-1}(e_n)=-e_1
\end{array}
\right.
\]
\item if $\tau=e_1+e_j$, then
\[
\left\lbrace \begin{array}{ll}
v^{-1}(e_i-e_{i+1})=e_{i+1}-e_{i+2} & \forall i\leq j-3\\
v^{-1}(e_{j-2}-e_{j-1})=e_1+e_{j-1}\\
v^{-1}(e_{j-1}-e_{j})=-e_1-e_{j}\\
v^{-1}(e_i-e_{i+1})=e_i-e_{i+1} & \forall i\geq j\\
v^{-1}(e_n)=e_n
\end{array}
\right.
\]

\end{enumerate}
\end{prop}

We showed that the orbits $(v,S)$ with $\#S=2$ are the only ones that admit two non isomorphic local systems.
Therefore, we want to study them more explicitly.

\begin{lemma}\label{ordineRMB}
Let $(v,S),(u,T)$ be admissible pairs such that $S=\left\lbrace e_1-e_j,e_1+e_j\right\rbrace$ and $T=\left\lbrace e_1-e_h,e_1+e_h\right\rbrace$.
Then $(v,S)\leq (u,T)$ if and only if $v\leq u$ and $j\geq h$.
\end{lemma}
\begin{proof}
Suppose $v\leq u$ and $j\geq h$.
Then we know that there is a sequence of ascents from $(v,S)$ to $(u,S)$, so we can suppose $u=v$.
Now, if $j=h$ we are done, so suppose $j>h$ and use induction on $j-h$.
There is $k\leq h<j$ $u=v=v_{\tau}$ such that $\tau=e_1+e_k$.
Then following the computations above $e_{j-1}-e_{j}=v^{-1}(e_{j-1}-e_{j})$.
Put $\gamma= e_{j-1}-e_{j}$.
Then $P_\gamma vx_S=Bvx_{s_\beta(S)}\cup Bvx_S \cup \bigcup Bvu_{-\beta}(t)x_S$.
But $s_\beta(S)=\left\lbrace e_1-e_{j-1},e_1+e_{j-1}\right\rbrace$ and $Bvx_{s_\beta}(S)$ is the open orbit in $P_\gamma vx_S$, so we conclude by induction.

We will now prove the converse.
Suppose $(v,S)\leq (u,T)$.
We know that this implies $v\leq u$, so we only need to show that $j\geq h$.
Note that the claim is clear when $v=u$.

Hence, suppose $v<u$ and by contradiction that $j<h$.
Then, if $e_1+e_k$ is the maximum root in $\Phi^+(u)$ it is clear that $h\geq k+2$.
By the calculations above, we know that $e_{k-1}-e_{k}$ is a complex descent on $u$ for $(u,T)$ and it is never a descent for $(v,S)$.
This means that we can inductively decrease the length of $u$ until $u=v$ which gives a contradiction.
\end{proof}

This lemma implies that, in this hypothesis, there is a sequence of descents from $(u,T)$ to $(v,S)$.

\begin{lemma}\label{rango2}
Let $\RM$ be the set of admissible pairs $(v,S)$ such that $\#S=2$.
Then given $(v,S)\leq (u,T)$ both in $\RM$ there is always a sequence of ascents between them.
Moreover, there is a minimum orbit $(v_{e_1+e_n},\left\lbrace e_1+e_n,e_1-e_n\right\rbrace)$.
\end{lemma}
\begin{proof}
Suppose $u\neq v_{e_1+e_n}$ and $T=\left\lbrace e_1-e_j,e_1+e_j\right\rbrace$.
Then, if $e_1+e_j$ is not the maximum root in $\Phi^+(u)$, we know that there is a descent $\alpha$ on $u$ which descends to the orbit $(s_\alpha u,T)$.
By the characterization above it must still be $(v,S)\leq (s_\alpha u,T)$, so we can suppose that $u=v_{e_1+e_j}$.
Then, take $\tau=e_j-e_{j+1}\in\Delta$.
We have $(u,s_{e_j-e_{j+1}}(T))\in\mathcal{E}_\tau(u,T)$ and unless $S=T$ we still have by the characterization above $(v,S)\leq (u,s_{e_j-e_{j+1}}(T)))<(u,T)$ and we are again in the previous case.
Inductively we obtain the claim.

The last statement follows because $v=v_{e_1+e_n}$ is the minimum element of $W^P$ such that $\Phi^+(v)$ contains two orthogonal roots and this roots are exactly $e_1-e_n$ and $e_1+e_n$.
\end{proof}

We saw that the orbits of maximum rank hold a special position also in the type $\bf{B}$ case.
Now, we want to study how the orbits in $\RM$ interacts with the other orbits in terms of the Bruhat order.
Note that the orbits with empty $\Psi$-part are minimal for dimensional reasons, so we only need to study the orbits which are represented by a single root.

As a first thing note that if $(v,S)$ is of maximum rank, hence $S=\left\lbrace e_1-e_h,e_1+e_h\right\rbrace$, then $(v,\left\lbrace e_1\right\rbrace)<(v,S)$.
For, we can see it as a matter of $B_{v}$-orbits in $\mathfrak{p}^u$ and then it is clear that there is $t\in\mathbb{C}$ such that
\[u_{e_h}(t).\left(e_{e_1-e_h}+e_{e_1+e_h}\right)=e_{e_1-e_h}+te_{e_1}\]
and given that $e_{1}-e_h$ and $e_1$ are linearly independent we get the thesis.
\begin{lemma}\label{e1}
Let $(v,S)$ be an orbit with $S=\left\lbrace \gamma\right\rbrace$, $\gamma\neq e_1$ and $\gamma^\perp$ the only root orthogonal to $\gamma$ in $\Psi$.
Fix $w\in W^P$ such that $w(e_1)<0$ and denote with $\beta$ the maximum root in $\Phi^+(w)$.
Then $\left(w,\left\lbrace e_1\right\rbrace\right)\leq(v,S)$ if and only if $w\leq v$ and $\beta<\gamma^\perp$.
\end{lemma}
\begin{proof}
Suppose $\beta<\gamma^\perp$ and $w\leq v$.
As a first thing note that $\gamma^\perp>\beta>e_1$, hence $\gamma<e_1<\gamma^\perp$
Now, if $\gamma^\perp\in\Phi^+(v)$ then there are descents between $(v,S)$ and $(v_{\gamma^\perp},S)$, so we can reduce the problem to two cases: $v=v_{\gamma^\perp}$ and $v<v_{\gamma^\perp}$.
Suppose the former and note that this implies $w<v$.
Then by Proposition \ref{azione WP}, if $\gamma^\perp=e_1+e_h$, the simple root $\alpha=e_{h-1}-e_h$ is such that $m_\alpha(v,S)=\left(v,\left\lbrace\gamma,\gamma^\perp\right\rbrace\right)$.
By what we said above $\left(v,\left\lbrace e_1\right\rbrace\right)<\left(v,\left\lbrace\gamma,\gamma^\perp\right\rbrace\right)$.
Now, $\mathcal{E}_\alpha\left(v,\left\lbrace e_1\right\rbrace\right)=\left(s_\alpha v,\left\lbrace e_1\right\rbrace\right)$, so $\left(s_\alpha v,\left\lbrace e_1\right\rbrace\right)\leq \left(v,S\right)$.
The claim follows because $w\leq s_\alpha v$ so there are ascents between $\left(w,\left\lbrace e_1\right\rbrace\right)$ and $\left(s_\alpha v,\left\lbrace e_1\right\rbrace\right)$.

For the latter, we can suppose without loss of generality that $w=v$.
Now, we know $\gamma$ is of the form $\gamma=e_1-e_j$.
Acting with $u_{e_1+e_j}(t)$ we get
\[u_{e_j}(t).e_{e_1-e_j}=e_{e_1-e_j}+te_1+t^2e_{e_1+e_j}\]
The last term can be cancelled acting with $u_{e_1+e_j}(s)$ because $v(e_1+e_j)>0$, so it is easy to see that $Bvx_{e_1}\subseteq \overline{Bvx_S}$.

Suppose now that $\left(w,\left\lbrace e_1\right\rbrace\right)\leq(v,S)$.
We certainly have $w\leq v$ and we know it must be $\sigma_{w(e_1)}\leq \sigma_{v(S)}$.
Note that if $\gamma^\perp\notin\Phi^+(v)$ the claim is clear, because the condition $\gamma^\perp\notin \Phi^+(w)\subseteq \Phi^+(v)$ is already true.
We can also suppose $\gamma=e_1-e_j$, because if $\gamma=e_1+e_j>e_1$ both sides of the implication are false.
Suppose at first $\beta\neq e_1$ and fix $w=v_{e_1+e_h}$ and $v=v_{e_1+e_k}$.

Then suppose $\gamma^\perp\in \Phi^+(v)$, or equivalently, $k\leq j$.
We basically want to show $h>j$.
If we compute the involutions we get 
\[\sigma_{w(e_1)}=s_{e_{h-1}}\]
while 
\[\sigma_{v(e_1-e_j)}=s_{e_{k-1}+e_j}\]
Let's try to impose $\sigma_{w(e_1)}\leq \sigma_{v(e_1-e_j)}$ and $h\leq j$.
Denote 
\[s_{e_{j}}=s_{e_j-e_{j+1}}\cdots s_{e_{n-1}-e_n}s_{e_n}s_{e_{n-1}-e_n}\cdots s_{e_j-e_{j+1}}\]
Then
\[s_{e_{k-1}+e_j}=s_{e_j}s_{e_{j-1}-e_j}\cdots s_{e_k-e_{k+1}}s_{e_{k-1}-e_{k}}s_{e_k-e_{k+1}}\cdots s_{e_{j-1}-e_j}s_{e_j}\]
and if we substitute $s_{e_j}$ with the above expression this is a reduced expression for $s_{e_{k-1}+e_j}$.
But $s_{e_{h-1}}$ commutes with all the factors of $s_{e_j}$ because $h-1<j$ and we get
\[s_{e_{h-1}}\leq s_{e_{j-1}-e_j}\cdots s_{e_k-e_{k+1}}s_{e_{k-1}-e_{k}}s_{e_k-e_{k+1}}\cdots s_{e_{j-1}-e_j}=s_{e_{k-1}-e_j}\]
which is impossible because $e_{k-1}-e_j$ is contained in the subsystem generated by the long roots in $\Delta$ which is clearly a system of type $\matr{A}_{n-1}$ and the reflection with respect to $e_{h-1}$ is not in the Weyl group of $\matr{A}_{n-1}$.
If $\beta=e_1$ we have $\sigma_{w(e_1)}=s_{e_n}$ and we can repeat the same reasoning.

\end{proof}

We will now define an order on $\mathcal{D}$.
Our intention is to prove later that this is equivalent to the Bruhat $\mathcal{G}$-order defined in \ref{ordinelb}.
\begin{defin}\label{preceq}
Let $(Bvx_S,\gamma),(Bux_R,\tau)\in\mathcal{D}$.
Then $(Bvx_S,\gamma)\preceq (Bux_R,\tau)$ if and only if $Bvx_S\leq Bux_R$ and one of the following is true:
\begin{enumerate}
\item both $\gamma$ and $\tau$ are trivial;
\item both $\gamma$ and $\tau$ are non-trivial;
\item $\#S\neq 2$ and $\gamma$ is trivial while $\tau$ is non-trivial;
\item $\#S=2$, $v<u$ and $\gamma$ is trivial while $\tau$ is non-trivial;
\item $\gamma$ is non-trivial while $\tau$ is trivial and there is $(w,T)\in\RM$ such that $Bvx_S< Bwx_T\leq Bux_R$ and $v<w$.
\end{enumerate}
\end{defin}

First of all, we need to prove that the one defined above is actually an order.
The only property that is difficult to prove is the transitive property.
So suppose we have 
\[(Bvx_S,\gamma)\preceq (Bux_R,\tau)\preceq (Bwu_T,\delta)\]
We want to prove that $(Bvx_S,\gamma)\preceq (Bwu_T,\delta)$.
If $\gamma$ and $\delta$ are either both trivial or non-trivial, then the claim follows from the transitivity of the Bruhat order.
So suppose at first $\gamma$ non-trivial and $\delta$ trivial.
Then either $\tau$ is trivial and there is $(z,V)\in\RM$ such that $Bvx_S< Bzx_V\leq Bux_R$ and $v<z$ or $\tau$ is non-trivial so there is $(z,V)\in\RM$ such that $Bux_R< Bzx_V\leq Bwu_T$ and $u<z$.
In both cases we obtain the claim.

If instead $\gamma$ is trivial while $\delta$ is non-trivial, we can suppose $\#S=2$ because if not it is easy to conclude thanks to the transitivity of the Bruhat order.

Similarly, if $\#S=2$ and $\tau$ is non-trivial we know that $v<u$, which implies $v<w$.

We are left with analysing the following situation:, $\gamma,\tau$ trivial, $\delta$ non-trivial and $\#S=2$.
Note that $\delta$ non-trivial implies also $\#T=2$.
If, by contradiction $v=u=w$, then it must be $\#R\neq 2$.
Given that $Bvx_{e_1}\leq Bvx_S$, by Lemma \ref{e1} we get $R=\left\lbrace e_1-e_h\right\rbrace$ and $e_1+e_h\notin\Phi^+(v)$.
This is absurd because the $\#T=2$ so both roots in $T$ must be bigger than $e_1-e_h$ from which follows that $(u,R)\nleq (u,T)$.

\begin{theorem}
If $\leq$ is the Bruhat $\mathcal{G}$-order defined in \ref{ordinelb}, then $\preceq$ is smaller than $\leq$.
\end{theorem}
\begin{proof}
We basically want to show that if $(Bvx_S,\gamma)\preceq (Bux_R,\tau)$, then $(Bvx_S,\gamma)\leq (Bux_R,\tau)$.
So suppose $Bvx_S\leq Bux_R$ and:
\begin{enumerate}
\item $\gamma$ and $\tau$ are trivial. 
This is Lemma \ref{ordtriv};
\item $\gamma$ and $\tau$ are non-trivial.
From Proposition \ref{propBlb} we know that $\#S=\#R=2$ so we can use Lemma \ref{rango2} to find a sequence of ascents between
$(v,S)$ and $(u,R)$.
Given that in every ascents the local system extends to a non-trivial local system we get the thesis;
\item $\#S\neq 2$ and $\gamma$ is trivial while $\tau$ is non-trivial.
Suppose $(u,R)$ is the smallest rank two orbit $(v_{e_1+e_n},\left\lbrace e_1+e_n,e_1-e_n\right\rbrace)$.
Then we know that $e_n$ is a descent for $(u,R)$ and the smallest orbit in $P_{e_n}ux_R $ is $(v_{e_1+e_n},\left\lbrace e_1\right\rbrace)$.
We are in case $c2)$ of Lemma \ref{listofaction}, so the non-trivial local system $\tau$ extends to all $P_{e_n}ux_R$ and the restriction to the smallest orbit is trivial.

Now look at $(v,S)$.
Either $e_n$ is a descent for $(v,S)$ or it isn't.
In the first case there is $(w,T)\in \mathcal{E}_{e_n}(v,S)$ such that $(w,T)\leq (v_{e_1+e_n},\left\lbrace e_1\right\rbrace)$.
If we endow these orbits with the respective trivial local systems $\gamma_1,\gamma_2$, then
\[\left((w,T),\gamma_1\right)\leq \left((v_{e_1+e_n},\left\lbrace e_1\right\rbrace),\gamma_2\right)\]
because of point $1)$ and $e_n$ is an ascent for both.
We have $\left(Bvx_S,\gamma\right)\in e_n\circ \left((w,T),\gamma_1\right)$ and $\left(Bux_R,\tau\right)\in e_n\circ \left((v_{e_1+e_n},\left\lbrace e_1\right\rbrace),\gamma_2\right)$, so the claim follows.

If $e_n$ isn't a descent for $(v,S)$, then $(v,S)\leq (v_{e_1+e_n},\left\lbrace e_1\right\rbrace)$.
Again if $\gamma_1,\gamma_2$ are the trivial local systems on these orbits we know
\[\left((v,S),\gamma_1\right)\leq \left((v_{e_1+e_n},\left\lbrace e_1\right\rbrace),\gamma_2\right)\]
and $\left((v_{e_1+e_n},\left\lbrace e_1\right\rbrace),\gamma_2\right)\leq \left(Bux_R,\tau\right)$

In general, we know that $(u,R)$ admits a complex descent $\alpha$.
So there is $(w,T)$ admissible and $\delta$ a non-trivial local system on the correspondent orbit such that $\left((u,R),\tau\right)\in\alpha\circ \left((w,T),\delta\right)$.
Then either $\alpha$ is a descent also for $(v,S)$ or not.
In the first case there is $(w',T')\in \mathcal{E}_\alpha(v,S)$ such that $(w',T')\leq (w,T)$.
Moreover $\#T'\neq 2$ (note that in a descent the cardinality of the $\Psi$-part never rises), so if $\delta'$ is the trivial local system on $(w',T')$ we get by induction that
\[\left((w',T'),\delta'\right)\leq \left((w,T),\delta\right)\]
which implies the thesis.

If instead $\alpha$ is not a descent for $(v,S)$ maintaining the notations above we obtain $(v,S)\leq (w,T)$ and because $\#S\neq 2$ while $\#T=2$ we know in fact that $(v,S)\neq (w,T)$, so we can apply induction and conclude;

\item $v<u$ and $\gamma$ is trivial while $\tau$ is non-trivial.
Suppose $\#S=2$ or we are in case above.
Consider $\mathcal{O}=(v,\left\lbrace e_1\right\rbrace)$ and $\mathcal{O}'=(u,\left\lbrace e_1\right\rbrace)$.
Then $\mathcal{O}<\mathcal{O}'$ because there is a sequence of ascents between them.
It follows that $(\mathcal{O},\gamma)<(\mathcal{O},\tau)$ when $\gamma$ and $\tau$ are both trivial local systems.
Then, by acting with $e_n$ we obtain

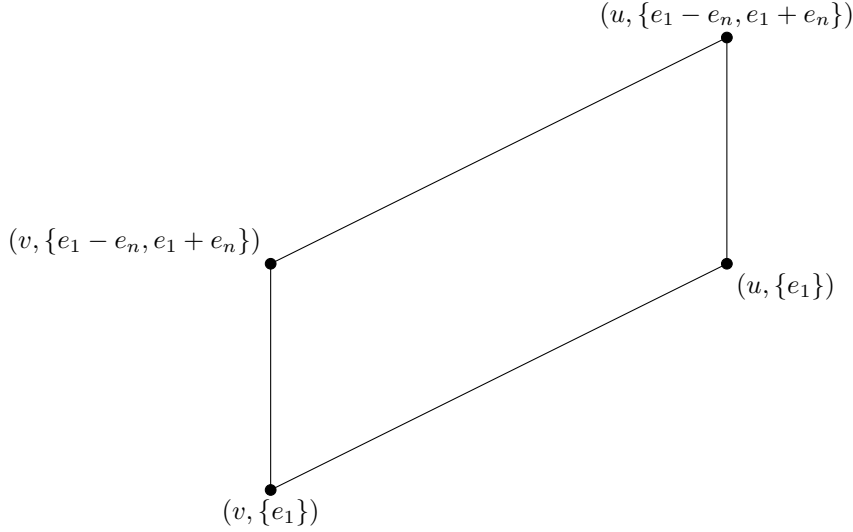
\begin{tikzpicture}
\filldraw[black] (8,0) circle (2pt) node[anchor=south] {$\left(u,\left\lbrace e_1-e_n,e_1+e_n\right\rbrace\right)$};
\filldraw[black] (8,-3) circle (2pt) node[anchor=north west]{$\left(u,\left\lbrace e_1\right\rbrace\right)$};
\filldraw[black] (2,-3) circle (2pt) node[anchor=south east]{$\left(v,\left\lbrace e_1-e_n,e_1+e_n\right\rbrace\right)$};
\filldraw[black] (2,-6) circle (2pt) node[anchor=north]{$\left(v,\left\lbrace e_1\right\rbrace\right)$};
\draw (2,-3) -- (8,0);
\draw (2,-6) -- (8,-3);
\draw (8,0) -- (8,-3);
\draw (2,-3) -- (2,-6);
\end{tikzpicture}

and we can choose to extend the trivial local systems trivially on $\left(v,\left\lbrace e_1-e_n,e_1+e_n\right\rbrace\right)$ and non-trivially on $\left(u,\left\lbrace e_1-e_n,e_1+e_n\right\rbrace\right)$.
By Lemma \ref{ordineRMB} we have $S=\left\lbrace e_1-e_h, e_1+e_h\right\rbrace$ and $T=\left\lbrace e_1-e_k, e_1+e_k\right\rbrace$ with $h\geq k$.
Using Proposition \ref{azione WP} we see that $(e_{n-1}-e_n,\ldots, e_{h}-e_{h+1})$ is a sequence of ascents from $\left(v,\left\lbrace e_1-e_n,e_1+e_n\right\rbrace\right)$ to $\left(v,S\right)$ while $(e_{n-1}-e_n,\ldots, e_{h}-e_{h+1},\ldots, e_k-e_{k+1})$ is a sequence of ascents from $\left(u,\left\lbrace e_1-e_n,e_1+e_n\right\rbrace\right)$ to $\left(u,T\right)$ and this concludes.

\item $\gamma$ is non-trivial while $\tau$ is trivial and there is $(w,T)\in\RM$ such that $Bvx_S\leq Bwx_T\leq Bux_R$ and $v<w$.

Let's suppose at first that $(u,R)\in\RM$ and that $v<u$.
Then consider $\left(v,\left\lbrace e_1\right\rbrace\right)$ and $\left(u,\left\lbrace e_1\right\rbrace\right)$.
It is
\[\left(v,\left\lbrace e_1\right\rbrace\right)\leq\left(u,\left\lbrace e_1\right\rbrace\right)\]
and if we endow them with the respective trivial local systems, this relation is true also in $\mathcal{D}$.

We can mirror the reasoning of point above with the only difference being that, once we have applied $e_n$ we consider the non-trivial local system on $\left(v,\left\lbrace e_1-e_n,e_1+e_n\right\rbrace\right)$ and the trivial one on $\left(u,\left\lbrace e_1-e_n,e_1+e_n\right\rbrace\right)$.


For a general $(u,R)$, just apply the previous reasoning to $(w,T)$ with trivial local system and then apply point $1)$.
\end{enumerate}
\end{proof}

As is, case $5)$ of definition \ref{preceq} is not easy to handle or to verify.
It turns out that we can restrict the search for a maximum rank orbit that fits in the middle to just one possible orbit.

Fix $(v,S)$ with $S=\left\lbrace \gamma\right\rbrace$ and $\beta$ the maximal element in $\Phi^+(v)$.
Define 
\[H=H(v,S)=\left\lbrace (u,R)\in\RM\mid (u,R)\leq (v,S)\right\rbrace\]

\begin{lemma}\label{maxorbRM}
With the notations given above:
\begin{enumerate}
\item $\beta<e_1+e_n \Rightarrow H=\varnothing$;
\item $\gamma\geq e_1-e_n\Rightarrow H=\varnothing$;
\item $\beta=e_1+e_j$ and $\gamma=e_1-e_h$ with $h<j$ imply that $H$ has a maximum which is 
\[(v,\left\lbrace e_1+e_j,e_1-e_j\right\rbrace)\]
\item $\beta=e_1+e_j$ and $\gamma=e_1-e_h$ with $j\leq h<n$ imply that $H$ has a maximum which is 
\[\left(v_{e_1+e_{h+1}},\left\lbrace e_1+e_{h+1},e_1-e_{h+1}\right\rbrace\right)\]

\end{enumerate}
\end{lemma}
\begin{proof}
We will show this case by case:
\begin{enumerate}
\item this is clear because if $(u,R)\leq (v,S)$ then $u\leq v$ and the maximum root in $\Phi^+(u)$ must be of the form $e_1+e_k\geq e_1+e_n>\beta$;
\item we can suppose that $\beta\geq e_1+e_n$.
If $\beta=\gamma$, then we are done for dimensional reasons.
So suppose $\beta>\gamma$ and, by contradiction, that $H\neq \varnothing$.
Then the minimum orbit $(v_{e_1+e_n},\left\lbrace e_1+e_n,e_1-e_n\right\rbrace)$ must be in $H$.
Now we can find $\alpha\in\Delta$ such that $s_\alpha v<v$ and $\alpha$ is a descent for $(v,S)$.
Moreover, Proposition \ref{azione WP} tells us that $\alpha$ is not a descent for $(v_{e_1+e_n},\left\lbrace e_1+e_n,e_1-e_n\right\rbrace)$.
It follows that if $\gamma\geq e_1+e_n$ we inductively obtain 
\[(v_{e_1+e_n},\left\lbrace e_1+e_n,e_1-e_n\right\rbrace)\leq \left(v_\gamma,\left\lbrace \gamma\right\rbrace\right)\]
which is again absurd for dimensional reasons.
If instead $\gamma=e_1-e_n$ or $\gamma=e_1$ we obtain
\[(v_{e_1+e_n},\left\lbrace e_1+e_n,e_1-e_n\right\rbrace)\leq \left(v_{e_1+e_n},\left\lbrace \gamma\right\rbrace\right)\]
which is again false for dimensional reasons;
\item as a first thing note that, with this hypothesis, $(u,R)\in H$ implies 
\[(u,R)\leq (v,\left\lbrace e_1+e_j,e_1-e_j\right\rbrace)\]
That's because $u\leq v$ and $(v,\left\lbrace e_1+e_j,e_1-e_j\right\rbrace)$ is the maximum among the orbits in $\RM$ with $W^P$-part equal to $v$.
Finally, with a simple computation we can prove that $(v,\left\lbrace e_1+e_j,e_1-e_j\right\rbrace)$ is in $H$;

\item 
The first thing to note is that if $(w,T)\in H$, then $w\leq v_{e_1+e_{h+1}}$ because it is easy to see that $\left(w,\left\lbrace e_1\right\rbrace\right)\leq (w,T)$ for every $T$ of cardinality two and then we apply Lemma \ref{e1}.
Given that $\left(v_{e_1+e_{h+1}},\left\lbrace e_1+e_{h+1},e_1-e_{h+1}\right\rbrace\right)$ is the maximum admissible pair with its $W^P$-part, we only need to show that it is in fact in $H$.

Suppose at first that $j=h$.
Then there is an ascent $\alpha\in\Delta$ such that $m_\alpha(v,S)=(w,T)=\left(v,\left\lbrace e_1+e_h,e_1-e_h\right\rbrace\right)$.
Now it is clear by the characterization of the order in $\RM$ that 
\[\left(v,\left\lbrace e_1+e_{h+1},e_1-e_{h+1}\right\rbrace\right)\leq(w,T)\]
Moreover, $\left(v,\left\lbrace e_1+e_{h+1},e_1-e_{h+1}\right\rbrace\right)$ descends through $\alpha$ to 
\[(z,V)\doteqdot\left(v_{e_1+e_{h+1}},\left\lbrace e_1+e_{h+1},e_1-e_{h+1}\right\rbrace\right)\]
which then is in $H$.

For $j<h$ it is enough to note that there is a sequence of ascents $\left(\alpha_1,\ldots,\alpha_n\right)$ between $\left(v_{e_1-e_h},S\right)$ and $\left(v,S\right)$.

\end{enumerate}
\end{proof}

Before continuing, it is useful to note that $e_n$ is an ascent only for orbits of the form $(v_{e_1-e_n},S)$, $\left(v_{e_1},\varnothing\right)$, $\left(v_\beta,\left\lbrace e_1\right\rbrace\right)$ or $\left(v_\beta,\left\lbrace e_1+e_n\right\rbrace\right)$ with $\beta\leq e_1+e_n$.

\begin{theorem}
The order $\preceq$ defined in \ref{preceq} and the Bruhat $\mathcal{G}$-order $\leq$ defined in \ref{ordinelb} are equivalent. 
\end{theorem}

\begin{proof}
We already proved that $\preceq$ is smaller than $\leq$, so it is enough to prove that $\preceq$ verifies the properties in definition \ref{ordinelb}.
So suppose $\left(Bvx_S,\gamma\right)\preceq \left(Bux_R,\tau\right)$ and fix $\alpha\in\Delta$ such that $\left(Bwx_T,\delta\right)\in \alpha\circ \left(Bux_R,\tau\right)$.
We want to show $\left(Bvx_S,\gamma\right)\preceq \left(Bwx_T,\delta\right)$.
This is the same as showing that $\left(Bux_R,\tau\right)\preceq \left(Bwx_T,\delta\right)$ because we already know that $\preceq$ is an order and so it is transitive.
The claim is clear if $\tau$ is non-trivial because it admits only non-trivial extensions or if both $\tau$ and $\delta$ are trivial.
So suppose $\tau$ trivial and $\delta$ non-trivial.
This can happen if and only if $Bux_R=\left(v,\left\lbrace e_1\right\rbrace\right)$ and $\alpha=e_n$.
But then $\#S\neq 2$, so we are in case $3)$ of definition \ref{preceq}, and we conclude.

Now, take $\left(Bvx_S,\gamma\right)\prec\left(Bux_R,\tau\right)$ and $\alpha\in\Delta$ such that $\left(Bzx_Z,\psi\right)\in \alpha\circ \left(Bvx_S,\gamma\right)$ and $\left(Bwx_T,\delta\right)\in \alpha\circ \left(Bux_R,\tau\right)$.
We want to show that $\left(Bzx_Z,\psi\right)\preceq\left(Bwx_T,\delta\right)$.
We need to analyse the situation case by case:

\begin{enumerate}
\item $\gamma$ and $\tau$ are trivial.
Then the claim is clear if both $\psi$ and $\delta$ are still trivial or if they are both non-trivial.
Hence, suppose $\psi$ non-trivial and $\delta$ trivial.
Then $\alpha=e_n$ and $S=\left\lbrace e_1\right\rbrace$.
But $\alpha$ is also an ascent for $Bux_R$ so it must be $(u,R)=\left(u,\left\lbrace e_1\right\rbrace\right)$ or $(u,R)=\left(u,\left\lbrace e_1+e_n\right\rbrace\right)$ because $u\geq v> v_{e_1}$.
In the first case it must be $v<u$ and given that the ascent doesn't change the $W^P$-part of the orbits we are in case $5)$ of Definition \ref{preceq} with $(w,T)\in\RM$.
The second case is impossible because of Lemma \ref{e1}.

Conversely, suppose that $\psi$ is trivial and $\delta$ is non-trivial.
Again we have $\alpha=e_n$ and in this case $R=\left\lbrace e_1\right\rbrace$.
Now, if $(v,S)=\left(v,\left\lbrace e_1\right\rbrace\right)$, then $v<u$ and we conclude as above.
If instead $(v,S)=\left(v,\left\lbrace e_1+e_n\right\rbrace\right)$, $\left(v,S\right)=\left(v_{e_1},\varnothing\right)$ or $\left(v,S\right)=\left(v_{e_1-e_n},S\right)$ then $\#Z\neq 2$ and we are in case $3)$ of Definition \ref{preceq}.

\item $\gamma$ and $\tau$ are non-trivial.
This is the easiest case because non-trivial local systems extend to non-trivial local systems;

\item $S\neq 2$ and $\gamma$ is trivial while $\tau$ is non-trivial.
Given that $\delta$ must be non-trivial we can suppose $\psi$ trivial. 
In this hypothesis the claim is clear if also $\#Z\neq 2$.
So suppose $\#Z=2$.
This can happen only if $S=\left\lbrace e_1\right\rbrace$ and $\alpha=e_n$ or if $\alpha$ is a real ascent for $(v,S)$.
The first case is absurd, because $e_n$ should be an ascent also for $(u,R)$ which is in $\RM$ and that's impossible.
Then the only possibility is that $Z=S\cup\left\lbrace \pm v^{-1}(\alpha)\right\rbrace$.
This implies that $(z,Z)$ is the maximum orbit with $W^P$-part equal to $z$, so it must be $z<w$.
We are then in case $4)$ of definition \ref{preceq} and the claim is proved;

\item $\#S=2$, $v<u$, $\gamma$ is trivial and $\tau$ is non-trivial.
Note that by the hypothesis it must be $(z,Z),(w,T)\in\RM$, so we only need to show that $z<w$.
Suppose by contradiction that $z=w$.
Then $\alpha$ is a descent for $(z,Z)$ and $(w,T)$.
We must have $s_\alpha z<z$ or we would have $v=u=z=w$.
But then it must be either $v=u$ or $u=w$.
The first case is clearly false and the second is also impossible because it would imply that $\alpha$ is a real ascent for $(v,S)$ which it can't be given that $(v,S)$ is of maximum rank;

\item $\gamma$ is non-trivial while $\tau$ is trivial and there is $(y,Y),\in\RM$ such that $Bvx_S< Byx_Y\leq Bux_R$ and $v<y$.
Note that we know that there is a maximum among the rank $2$ orbit that sit between $(v,S)$ and $(u,R)$, se we can suppose $(y,Y)$ is that maximum.
If $\alpha$ is also an ascent for $(y,Y)$ the claim is clear.
For, suppose $(y',Y')=m_\alpha(y,Y)$, then $Bzx_Z<By'x_{Y'}\leq Bwx_T$, $(y',Y')\in \RM$ and $z<y'$.

If $\alpha$ is not an ascent for $(y,Y)$, then we still have that $Bzx_Z\leq Byx_Y\leq Bwx_T$, but we could theoretically have $z=y$.
So we need to study only this final case.
Note that by the maximality of $Byx_Y$ and our hypothesis that $\alpha$ is not an ascent for $(y,Y)$ we know that $\#R\neq 2$.

Our hypotheses force $v<s_\alpha v=z=y$.
Moreover, by Lemma \ref{maxorbRM} we can't have $y=u$ because $\alpha$ must be an ascent for $(u,R)$.
It follows that $y<u$ and we are in case $4)$ of Lemma \ref{maxorbRM}.
Then, $\alpha$ must act on $(u,R)$ as a complex ascent on $R$ (note that if $s_\alpha v>v$ it can't be $s_\alpha u>u$ because of Proposition \ref{azione WP}).
So $m_\alpha(u,R)=(u, s_{u^{-1}(\alpha)}(R))$.
Now it is easy to see that if $R=\left\lbrace e_1-e_j\right\rbrace$, then $s_{u^{-1}(\alpha)}(R))=\left\lbrace e_1-e_{j-1}\right\rbrace$, hence, by Lemma \ref{maxorbRM}, the maximum orbit of rank $2$ contained in $m_\alpha(u,R)$ is strictly bigger than $(y,Y)$ and in particular it has strictly bigger $W^P$-part and that let us conclude.
\end{enumerate}
\end{proof}

The results in this section let us describe completely the Bruhat $\mathcal{G}$-order in $\mathcal{D}$ for the type $\matr{B}$ case.

\begin{theorem}\label{GorderB}
Let $(Bvx_S,\gamma),(Bux_R,\tau)\in\mathcal{D}$.
Then $(Bvx_S,\gamma)\leq (Bux_R,\tau)$ if and only if $Bvx_S\leq Bux_R$ and one of the following is true:
\begin{enumerate}
\item both $\gamma$ and $\tau$ are trivial;
\item both $\gamma$ and $\tau$ are non-trivial;
\item $\#S\neq 2$ and $\gamma$ is trivial while $\tau$ is non-trivial;
\item $\#S=2$, $u<v$ and $\gamma$ is trivial while $\tau$ is non-trivial;
\item $\gamma$ is non-trivial while $\tau$ is trivial, $H(u,R)\neq \varnothing$ and $(u',R')=\max H(u,R)$ verifies
$(v,S)\leq (u',R')$ with $v<u'$.
\end{enumerate}
\end{theorem}

\section{The type \bf{C} case}\label{C}

We now want to study the case where the root system $\Phi$ is of type $\bf{C}$.
A root system of type $\bf{C}_n$ can be realized as
\[\Phi=\left\lbrace \pm e_i\pm e_j\right\rbrace_{i,j=1,\ldots,n}\cup \left\lbrace \pm 2e_i\right\rbrace_{i=1,\ldots,n}\subseteq \langle e_1,\ldots, e_n\rangle_{\mathbb{R}}\cong \mathbb{R}^n\]
The usual choice for a basis is
\[\Delta=\left(\underbrace{e_1-e_2}_{\alpha_1},\ldots,\underbrace{e_{n-1}-e_n}_{\alpha_{n-1}},\underbrace{2e_n}_{\alpha_n}\right)\]
which corresponds to 
\[\Phi^+=\left\lbrace e_i\pm e_j\right\rbrace_{i<j}\cup\left\lbrace 2e_i\right\rbrace_i\]

The abelianity of the unipotent radical forces the parabolic $P$ to be the one associated to $\Delta\setminus \left\lbrace 2e_n\right\rbrace$, so
\[\Psi=\left\lbrace e_i+e_j\right\rbrace_{i<j}\cup\left\lbrace 2e_i\right\rbrace_i\]

To fix our ideas we will use $\matr{SP}_{2n,\mathbb{C}}$ as a concrete example of simply connected linear algebraic group of type $\bf{C}$.
The following lemma assures us that this case is not trivial.

\begin{lemma}
Let $G$ be a simply connected linear algebraic group of type $\bf{C}$ and $G/L$ a Hermitian symmetric variety.
Let $(v,S)$ be an admissible pair that represents a B-orbit $\mathcal{O}$ on $G/L$.
Then the set of isomorphism classes of $B$-equivariant local systems on $\mathcal{O}$ has $2^k$ elements where $k$ is the number of long roots in $S$.
\end{lemma}
\begin{proof}
We can assume without loss of generality that $G=\matr{SP}(2n,\mathbb{C})$.
In this case we can choose the torus as the subgroup of diagonal matrices
\[\diag(t_1,\ldots,t_n,t_n^{-1},\ldots,t_1^{-1})=\left(\begin{array}{ccc|ccc}
t_1 & & & & &\\
& \ddots & & & & \\
& &t_n & & &\\
\hline
& & &t_n^{-1} & &\\
& & & & \ddots &\\
& & & & & t_1^{-1}
\end{array}
\right)
\]
This is quite convenient because 
\begin{align*}
\left(e_i+e_j\right)\left(\diag(t_1,\ldots,t_n,t_n^{-1},\ldots,t_1^{-1})\right)&=t_it_j\\ 
2e_i\left(\diag(t_1,\ldots,t_n,t_n^{-1},\ldots,t_1^{-1})\right)&=t_i^2
\end{align*}
Then, $\stab_T(vx_S)\cong \stab_T(x_S)=\left\lbrace t\in T\mid \gamma(t)=1\text{ }\forall \gamma\in S\right\rbrace$.
We can think of the torus $T$ as $\left(\mathbb{C}^*\right)^n$.
Then every short root represent a relation $t_i=t_j^{-1}$ which decreases the dimension by one.
On the contrary, every long root represent a relation $t_i^2=1$ which has two solutions $t_i=1$ and $t_i=-1$.
It follows that if $S$ contains $r$ long roots and $s$ short roots, the stabilizer is isomorphic to
\[\stab_T(x_S)\cong 2^r\times \left(\mathbb{C}^*\right)^{(n-r-s)}\]
The group of connected components $\pi_0\left(\stab_T(x_S)\right)$ is then isomorphic to $\left(\mathbb{Z}/2\mathbb{Z}\right)^r$.

\end{proof}

The above lemma gives us a way to identify the isomorphism classes of local systems.
Recall that this classes are in a natural one to one correspondence with the continuous representation of $\stab_B(vx_S))$ on $\mathbb{C}$ which themselves correspond to representations of $\pi_0(\stab_T(x_S))$.
Suppose $S_l=\left\lbrace 2e_{i_1},\ldots,2e_{i_r}\right\rbrace$ and for every $j=1,\ldots,r$ define 
\[L_{i_j}=\diag\left(t^{i_j}_1,\ldots,t^{i_j}_n,\left(t^{i_j}_n\right)^{-1},\ldots, \left(t_1^{i_j}\right)^{-1}\right)\]
 $t^{i_j}_k=1$ for every $k\neq i_j$ and $t^{i_j}_{i_j}=-1$.
Then all the $L_{i_j}$ live in different connected components and their connected components generate $\pi_0(\stab_T(x_S))$.
A representation is then a map $\pi\colon\pi_0(\stab_T(x_S))\longrightarrow\matr{GL}(\mathbb{C})=\mathbb{C}^*$ and it is clear that is defined by the images of  $L_{i_j}$.
Note that these have order $2$, so $\pi(L_{i_j})=\pm 1$.
Then a local system over $Bvx_S$ can be represented as a sequence $\left(a_1,\ldots, a_r\right)$ defined by $a_j\doteqdot\pi(L_{i_j})$.

By what we said in section \ref{linebundle} we are in case $c1)$ of Lemma \ref{listofaction} if and only if there is a root $\gamma\in S$ such that $\beta=v^{-1}(\alpha)$ can be both added and subtracted to $\gamma$.
This can happen in type $\bf{C}$ with, for example $\gamma=e_i+e_{i+1}$ and $\beta=e_i-e_{i+1}$.
Suppose that we have an admissible pair $(v,S)$ and an ascent $\alpha$ such that $S$ contains a root of the form $e_i+e_{i+1}$ and $\beta=v^{-1}(\alpha)=e_i-e_{i+1}$.
Then $(v,S)\xmapsto{\alpha}\left(v,\left(S\setminus\left\lbrace \gamma\right\rbrace \right)\cup\left\lbrace 2e_i,2e_{i+1}\right\rbrace\right)$.
The first thing to note is that if $S$ contained exactly $r$ different long roots, then $S'=\left(S\setminus\left\lbrace \gamma\right\rbrace \right)\cup\left\lbrace 2e_i,2e_{i+1}\right\rbrace$ contains $r+2$ different long roots.
So, while $(v,S)$ admitted $2^r$ non-isomorphic local systems, $(v,S')$ admits $4$ times that, even though every local system on $Bvx_S$ can be extended in only two different ways.
It is then natural to ask which local system of $(v,S')$ comes from a local system of $(v,S)$.

\begin{lemma}
Fix $(v,S)$ an admissible pair and suppose there is an ascent $\alpha\in \Delta$ for $(v,S)$ that realizes case $c1)$ of Proposition \ref{listofaction}.
Denote $(v',S')=m_\alpha(v,S)$.
Suppose that $X=(\pm 1,\ldots,\pm 1)$ is the sequence associated to a local system $\tau$ on $Bvx_S$ as explained above.
Then $\left(S'\right)_l=S_l\cup\left\lbrace \delta_1,\delta_2\right\rbrace$ and the sequences associated to the two possible extension of $\tau$ to $Bv'x_{S'}$ are obtained from $X$ by adding in the positions relative to $\delta_1$ and $\delta_2$ either two $1$ or two $-1$.
\end{lemma}
\begin{proof}
For this, we will need to shift our perspective a bit.
Since the beginning, we studied the $B$-orbits in $G/L$ because we had a more concrete understanding of them thanks to \cite{GM}.
On the other hand, Lusztig and Vogan (\cite{Vogan} and \cite{LV}) always refer to $L$-orbits on the flag variety $B\backslash G$.
Right now we will refer to \cite{LV}, so we will use this point of view.
As we said in the second section there is a one to one correspondence between $B$-orbits in $G/L$, $B\times L$-orbits in $G$ and $L$-orbits in $B\backslash G$.
The correspondence is
\[BxL/L\leftrightsquigarrow BxL\leftrightsquigarrow B\backslash BxL\]

In the $G/L$ setting we used the minimal parabolic subgroups to study the orbits.
If $\alpha\in\Delta$ and $\mathcal{O}_B=BxL/L$ is a $B$-orbit in $G/L$, $P_\alpha\mathcal{O}_B$ is a finite union of orbits that we know quite well.
The analogous in the $G/B$ setting is the union of $\alpha$-lines.
If $B\backslash Bx\in\ B\backslash G$ is an element of the flag variety (which can be seen as a Borel subgroup of $G$), then the $\alpha$-line through $B\backslash Bx$ is $B\backslash P_\alpha x$.
If $\mathcal{O}_L$ is the $L$-orbit that correspond to $\mathcal{O}_B$, then $P_\alpha\mathcal{O}_B$ corresponds
\[P_\alpha \mathcal{O}_B\leftrightsquigarrow\bigcup_{B\backslash By\in\mathcal{O}_L}B\backslash P_\alpha y=B\backslash P_\alpha xL\]
Moreover, the $L$ orbits in $B\backslash P_\alpha xL$ are in correspondence with the $L'=L\cap x^{-1}P_\alpha x$-orbits in $B\backslash P_\alpha x$.
The advantage of this point of view is that $B\backslash P_\alpha x\cong P_\alpha/B$ is isomorphic to $\mathbb{P}^1$.

The $B$-equivariant local systems on $\mathcal{O}_B$ are in a natural correspondence with the $B\times L$-equivariant local systems on the associated $B\times L$-orbit $\mathcal{O}_{B\times L}$ in $G$ and with the $L$-equivariant local systems on $\mathcal{O}_L$.
All in all, we have the following chain of correspondences
\[
\begin{array}{lcc}
\hom\left(\pi_0\left(\stab_B(xL/L)\right),\mathbb{C}^*\right) & \Leftrightarrow & \left\lbrace B\text{-equivariant local system on }BxL/L\right\rbrace\\
& & \big\Updownarrow \\
\hom\left(\pi_0\left((\stab_L(B\backslash Bx)\right),\mathbb{C}^*\right)& \Leftrightarrow  & \left\lbrace L\text{-equivariant local systems on }B\backslash BxL\right\rbrace 
\end{array}
\]

The correspondence between $\hom\left(\pi_0\left(\stab_B(xL/L)\right),\mathbb{C}^*\right) $ and $\hom\left(\pi_0\left((\stab_L(B\backslash Bx)\right),\mathbb{C}^*\right)$ can be made explicit.
For, note that $\stab_B(xL/L)=B\cap xLx^{-1}$ while $\stab_L(B\backslash Bx)=x^{-1}Bx\cap L$ so $\stab_B(xL/L)=x\stab_L(B\backslash Bx)x^{-1}$.

Note that because of equivariancy, the $L$-equivariant local systems on $B\backslash BxL$ correspond to the $L'$-equivariant local systems on $B\backslash P_\alpha x$ where $L'=L\cap x^{-1}P_\alpha x$.
Moreover, this correspondence commute with the fact that $\stab_L\left(B\backslash Bx\right)=\stab_{L'}\left(B\backslash Bx\right)$.
So, our sequence $X$ gives us a map $\phi\colon\pi_0\left(\stab_L(B\backslash Bx)\right)\longrightarrow\mathbb{C}^*$ which corresponds to a map 
\[\phi:\pi_0\left(\stab_{L'}\left(B\backslash Bx\right)\right)\longrightarrow\mathbb{C}^*\]

Now, we noted above that $B\backslash P_\alpha x\cong \mathbb{P}^1$ and the group of automorphisms of $\mathbb{P}^1$ is $\matr{PSL}_2$.
Given that $L'$ acts on $\mathbb{P}^1$, there must be a map
\[\psi:L'\longrightarrow \matr{PSL}_2\]
The image of $L'$ must have two orbits in $\mathbb{P}^1$.
We can then suppose without loss of generality that 
\[\psi(L')=\N(T)=\left\lbrace \left(\begin{array}{cc}
t & \\
 & t^{-1}
\end{array}
\right)\right\rbrace \cup \left\lbrace \left(\begin{array}{cc}
 & t\\
-t^{-1} & 
\end{array}
\right)\right\rbrace\]
where $T$ is the torus made by the diagonal matrices.
If we denote $s=\left(\begin{array}{cc}
0 & -1\\
1 & 0
\end{array}\right)$ then $\N(T)=T\cup sT$.
Here $\left(\begin{array}{cc}t & \\
& t^{-1}
\end{array}\right)$ act as multiplication by $t^2$ while $\left(\begin{array}{cc}
0 & t \\
-t^{-1} & 0
\end{array}\right)$ act as $t^{-2}$.
In this case it is an easy computation to see that there are two orbits in $\mathbb{P}^1$: one that contains $0$ and the point at infinity and the other that contains every remaining element and is open.

We will now briefly study the $\N(T)$-equivariant local systems on $\mathbb{P}^1$.

To start, every local system on $\mathbb{P}^1$ is trivial (but they may not be trivial as equivariant local systems), so we can write the local system on $\mathbb{P}^1$ as a projection on the first coordinate $\pi_1\colon\mathbb{P}^1\times \mathbb{C}\longrightarrow \mathbb{P}^1$.
To make it $\N(T)$-equivariant we need to define a map for every $f\in\N(T)$ which we will denote again with $f$

\begin{align*}
f:\mathbb{P}^1\times \mathbb{C}&\longrightarrow\mathbb{P}^1\times \mathbb{C}\\
(x,z)&\longmapsto f(x,z)=\left(f_1(x,z),f_2(x,z)\right)
\end{align*}

and it must be $f_1(x,z)=f.x$ where with $f.x$ we denote the action of $\N(T)$ on $\mathbb{P}^1$ that we briefly described above.
Moreover, for any fixed $x\in\mathbb{P}^1$, $f_2(x,z)\colon\mathbb{C}\longrightarrow\mathbb{C}$, so it must be of the form $f_2(x,z)=f_2(x)z$.
On the other hand $f_2\colon\mathbb{P}^1\longrightarrow\mathbb{C}^*$ must again be constant, so
\[f(x,z)=\left(f.x,\rho(f)z\right)\]
where $\rho\colon\N(T)\longrightarrow\mathbb{C}^*$.
Note that if we fix an $x\in\mathbb{P}^1$ and we restrict $\rho$ to $\stab_{\N(T)}(x)$, then $\rho$ is exactly the representation that is uniquely associated to the restriction of our local system to the orbit of $x$.
Moreover it is essential to note that if we have a local system on the closed orbit, then it extends to the open orbit uniquely, simply because it depends only on the $\rho$ which does not depend on the points $x$ of $\mathbb{P}^1$.
But, this does not mean that two isomorphic (trivial) local systems on the closed orbit can not extend to different local systems on the whole $\mathbb{P}^1$ as we will see soon.

It follows that we can write $\tau_L$ restricted to $\mathbb{P}^1$ as the trivial bundle $\left\lbrace 0,\infty\right\rbrace\times \mathbb{C}$.
On this we may let $f\in\stab_{L'}(0)$ act as $f(0,z)=(0,\rho(f)z), f(\infty,z)=(\infty,\rho(f)z)$ and $s.(0,z)=(\infty,\rho(s)z)$
$s.(\infty,z)=(0,\rho(s)z)$ and $s^2=\id$, so $\rho(s)=\pm 1$.
Different choices of $\rho(s)$ give different local systems on $\mathbb{P}^1$ and, in turn, different extensions of $\tau_L$.

Let's see what this means for the stabilizers.
We know that the stabilizer in $\matr{PSL}_2$ of $0$ is $T$, so $\stab_L(B\backslash Bx)\cong \ker(\psi)T$.
On the other hand $\stab_{\matr{PSL}_2}(1)=\langle s\rangle$, so if $y\in B\backslash B$ corresponds to $1$ we get $\stab_{L}(B\backslash By)=\ker(\psi)\cup s\ker(\psi)$.

To end, we need to pull everything back to the $G/L$ setting.
We know that $x=vx_S$ and $y=vx_{R}$ with $R=\left(R\cap S\right)\cup \left\lbrace 2e_i,2e_{i+1}\right\rbrace$ and $S=\left(R\cap S\right)\cup \left\lbrace e_i+e_{i+1}\right\rbrace$.
We also have a map $\phi\colon\stab_T(x_SL/L)\longrightarrow\mathbb{C}^*$ that corresponds to our sequence $X$.
We want to understand what are the two possible maps $\phi'\colon\stab_T(x_R/L)\longrightarrow\mathbb{C}^*$ that corresponds to extensions of $\tau$ to $Bvx_S$.
We know that it is enough to compute the value $\phi(t)$ for $t=\diag(1,\ldots, -1,\ldots,1)$ where we have a single $-1$ in position $j$ if $2e_j\in R$.
Note that if $j\neq i,i+1$ then $t$ is also in $\stab_T(x_SL/L)$.

Let's suppose we are in this case, so $t\in \stab_T(x_SL/L)\cap \stab_T(x_RL/L)$.
The first thing to do is multiplying by $v$ on the left and $v^{-1}$ on the right to obtain $vtv^{-1}\in\stab_T(vx_RL/L)$.
Then we multiply by $\left(vx_R\right)^{-1}$ on the left and $vx_R$ on the right to obtain $x_R^{-1}tx_R\in\stab_L(B\backslash Bvx_R)$.
Note that $t\exp\left(\sum e_{\alpha_i}\right)=\exp\left(\sum \alpha_i(t)e_{\alpha_i}\right)$.
It follows that $t\in \stab_T(x_RL/L$ if and only if $tx_R=x_Rt$, hence $x_R^{-1}tx_R=t$.
For the same reason $x_Stx_S^{-1}=t\in\stab_L(B\backslash Bvx_S)$.
But then $t\in\stab_L(B\backslash Bvx_R)\cap\stab_L(B\backslash Bvx_S)=\ker(\psi)$, so $\phi'(t)=\phi(t)$

We are now left to compute the value of $\phi'(t)$ where $t$ has $-1$ in position $i$ or $i+1$.
This values must be concordant given that if $t_i$ is the diagonal matrix with $-1$ in position $i$ and $1$ everywhere else, then $t_it_{i+1}\in\stab_T(x_SL/L)\cap \stab_T(x_RL/L)$ and $\phi(t_it_{i+1})=1$ because it is in the connected component of the identity.
On the other hand we know that there are two possible definition of $\phi'$, so one of them must be $\phi'(t_i)=\phi'(t_{i+1})=1$ and the other $\phi'(t_i)=\phi'(t_{i+1})=-1$.
\end{proof}

We also have the following results for other type of ascents.
\begin{lemma}
Let $(v,S)$ be an admissible pair with $\#S_l=r$ and let $X=(\underbrace{1,\ldots,-1,\ldots,1}_r)$ be the sequence associated to a local system on $Bvx_S$.

Let $\alpha\in\Delta$ be an ascent for $(v,S)$.
Then
\begin{enumerate}
\item if $\alpha$ is of type $b1)$ then the local system extends uniquely to the open orbit of $P_\alpha vx_S$ with sequence
equal to $X$;

\item if $\alpha$ is of type $d1)$ then the local system extends uniquely to the open orbit of $P_\alpha vx_S$ with sequence
$(\underbrace{1,\ldots,-1,\ldots,1}_{X},\underbrace{1}_{r+1})$.
\end{enumerate}
\end{lemma}

While we don't have a complete characterization of the order in $\mathcal{D}$ as we have for the simply laced case and for type $\bf{B}$, there is something we can say about the connected components of the Hasse diagram.
This in turn will give a necessary condition for the order in $\mathcal{D}$.
To show this, we need to study the sequences that determine possible local systems.

\begin{defin}
Let $X=(a_1,\ldots,a_r)$ a sequence where $a_i=\pm 1$.
Then we define the \textit{plus number} of $X$ as 
\[\pl(X)=\sum_{a_i=1}(-1)^i\]
 and the \textit{minus number} of $X$ as
 \[\mi(X)=\sum_{a_i=-1}(-1)^i\]
\end{defin}

The most important property of the minus and plus numbers is the following:
\begin{prop}\label{above2}
Let $X=\left(a_1,\ldots,a_r\right)$ as above and fix $j\in\left\lbrace 1,\ldots, r\right\rbrace$ such that $a_j=a_{j+1}$.
Denote with $Y$ the sequence $\left(a_1,\ldots, \widehat{a_j},\widehat{a_{j+1}},\ldots,a_r\right)$.
Then $\pl(Y)=\pl(X)$ and $\mi(Y)=\mi(X)$.
\end{prop}
\begin{proof}
To avoid confusion we will write $Y=\left(b_1,\ldots,b_{r-2}\right)$ with $b_i=a_i$ for every $i<j$ and $b_i=a_{i+2}$ for every $i\geq j$.

Suppose $a_j=1$.
We have 
\begin{align*}
\pl(X)=\sum_{a_i=1}(-1)^i&=\sum_ {a_i=1, i<j}(-1)^i+\sum_{a_i=1,i>j+1}(-1)^i+(-1)^j+(-1)^{j+1}\\
&=\sum_ {a_i=1, i<j}(-1)^i+\sum_{a_i=1,i>j+1}(-1)^i\\
&=\sum_ {b_i=1, i<j}(-1)^i+\sum_{b_i=1,i\geq j}(-1)^{i}=\pl(Y)
\end{align*}
where the last equality follows because for $(-1)^i=(-1)^{i+2}$.
On the other hand
\[\mi(X)=\sum_{a_i=-1}(-1)^i=\sum_ {a_i=-1, i<j}(-1)^i+\sum_{a_i=-1,i>j+1}(-1)^i=\mi(Y)\]
as above.
The case $a_j=-1$ is symmetric.
\end{proof}

We showed that the plus and minus numbers don't change if we delete two identical adjacent numbers.

Given a sequence, we can delete all pairs of identical adjacent numbers until it is no longer possible to do so.
Doing so, by Lemma \ref{above2}, the plus and minus numbers don't change.
The final sequence must be one of alternating signs and it is easy to see that the sequences of alternating signs are uniquely identified by their plus and minus numbers.
Hence, the final sequence doesn't depend on the order in which we delete the pairs of identical adjacent numbers.
There is another thing we can do; if the rightmost element of the final sequence is $1$ we delete it.

\begin{defin}
Let $X=\left(a_1,\ldots,a_r\right)$ be a sequence with $a_i=\pm 1$.
Denote with $\re(X)$ the sequence obtained by $X$ deleting inductively all pairs of identical adjacent number and then, if present, the rightmost $1$.

The sequence $\re(X)$ is called the \textit{reduced form} of $X$.
\end{defin}

We can finally state our last result.
\begin{theorem}\label{GorderC}
Let $X$ and $Y$ be the sequences associated respectively to a local system on $(v,S)$ and $(u,R)$.
Then the corresponding elements in $\mathcal{D}$ are in the same connected component of the Hasse diagram if and only if $\re(X)=\re(Y)$.
\end{theorem}
\begin{proof}
We will start by showing that if $\left((v,S),X\right)$ and $\left((u,R),Y\right)$ are in the same connected component, then $\re(X)=\re(Y)$.
To see this, consider the following order 
\[\left((v,S),X\right)\prec\left((u,R),Y\right)\Leftrightarrow(v,S)< (u,R)\text{ and }\re(X)=\re(Y)\]
it verifies the conditions of definition \ref{ordinelb} because we saw above that extending with an ascent doesn't change the reduced form.
This means that $\left((v,S),X\right)\leq\left((u,R),Y\right)$ implies $\re(X)=\re(Y)$ and the claim follows.

We will now show that if $\re(X)=\re(Y)$ then $\left((v,S),X\right)$ and $\left((u,R),Y\right)$ are in the same connected component.
Recall that every orbit that is not the open orbit in $G/L$ admits an ascent.
This implies that for every orbit $\mathcal{O}$ there is a sequence of ascents between $\mathcal{O}$ and the open orbit.
If we have an isomorphism class of local systems on $\mathcal{O}$ we can get a sequence of ascents in $\mathcal{D}$ by choosing one of the possible extensions in every ascent.
We know that doing this doesn't change the reduced form, so we can suppose without loss of generality that $Bvx_S=Bux_R$ are the open orbit of $G/L$ .
More precisely, $v=\omega_P$ the longest element in $W^P$ and $R=S=\left\lbrace 2e_1,\ldots,2e_n\right\rbrace$.

Now it is easy to see that $v(e_i)=-e_{n-i}$ for every $i=1,\ldots,n$.
Hence, if we fix $j\in\left\lbrace 1,\ldots,n-1\right\rbrace$ we have $v(e_j-e_{j+1})=e_{n-j-1}-e_{n-j}\in\Delta$.
It follows that for every $j$ there is a descent $\alpha\in\Delta$ such that $(v,T)\in\mathcal{E}_\alpha(v,S)$ with $T=\left\lbrace 2e_i\right\rbrace_{i\neq j,j+1}\cup\left\lbrace e_{j}+e_{j+1}\right\rbrace$.
Moreover, for $(v,T)$ the root $\alpha$ is an ascent of type $c1)$.
To see what all this means, take $X=\left(a_1,\ldots,a_n\right)$ and $j\in\left\lbrace 1,\ldots,n-1\right\rbrace$ such that $a_j=a_{j+1}$.
Denote with $X'=\left(b_1,\ldots, b_n\right)$ the sequence with $b_i=a_i$ for every $i\neq j,j+1$ and $b_j=-a_j$, $b_{j+1}=-a_{j+1}$ and with $X_0=\left(c_1,\ldots,c_{n-2}\right)$ the sequence with $c_i=a_i$ for every $i<j$ and $c_i=a_{i+2}$ for $i\geq j$.
If we take the simple root $\alpha$ as above we have
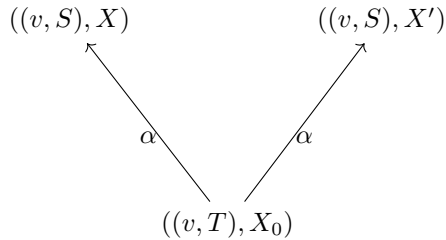
\begin{center}
\begin{tikzpicture}
\node (A) at (0,0)[anchor=west] {$\left((v,S),X\right)$};
\node (B) at (6,0)[anchor=east] {$\left((v,S),X'\right)$};
\node (C) at (3,-3)[anchor=south] {$\left((v,T),X_0\right)$};
\draw [->] (C) -- node[midway,below]{$\alpha$} (A);
\draw [->] (C) -- node[midway,below]{$\alpha$}(B);
\end{tikzpicture}
\end{center}
It follows that $(Bvx_S,X)$ and $(Bvx_S,X')$ are in the same connected component.
This can clearly be done every time we have a pair of identical adjacent element in $X$.

We now want to show an algorithm that takes a sequence $X$ and gives another sequence that depends only on $\re(X)$.
Every step of the algorithm will be inverting the sign of two adjacent identical element, hence the final sequence will be in the same connected component of $X$ and thus the claim will be proved.

The algorithm is the following.
If there is $k\in\left\lbrace 1,\ldots,n\right\rbrace$ such that $a_i=1$ for every $i\leq k$ while $a_i=-a_{i+1}$ for every $i\geq k$ then we are done.
If not, let $k$ be the smallest number such that either $a_{k-1}\neq a_k=a_{k+1}$ or $a_k=a_{k+1}=-1$.
Change the sign of $a_k$ and $a_{k+1}$ and repeat the algorithm with the new sequence.

Note that there is always a $k$ with the property above unless we are in the first case.
Moreover, the algorithm ends.
To see this denote with $s_0$ the length of the initial sequence of $1$ in $X$ and with $s_i$ the same length after $i$ steps of the algorithm.
Moreover, denote with $k_i$ the element $k$ defined above for the $i$-th step.
We want to show that after every step either $k_i$ decreases or $s_i$ increases.
Given that they can't decrease infinitely or increase infinitely we would have the claim.
So, fix a sequence $X$ and the relative $s=s_0$ and $k=k_0$.
If $k$ is such that $a_{k-1}\neq a_k=a_{k+1}$ we either have $s=k-1$ and $a_{k-1}=1$ or $s<k-1$.
In the first case it is clear that $s_1\geq k+1>s$ while in the second case $s_1=s$ and $k_1=k-1<k$.
Suppose now that $k$ is such that $a_k=a_{k+1}=-1$.
Then it must be $s=0$ and $k=1$ or we would be in the previous case, which means that $s_1\geq 2>s$.

The last thing to note is that the algorithm doesn't change the reduced form of the sequence because it changes only adjacent pair of identical values and the output of the algorithm clearly depends only on the reduced form $X$.
This proves the claim.

\end{proof}
\bibliographystyle{unsrt}
\bibliography{bibartsisloc}{}

\begin{thebibliography}{1}

\bibitem{RS2}
R.W. Richardson and T.A. Springer.
\newblock {Combinatorics and geometry of K-orbits on the flag manifold}.
\newblock In R.S Elman, M.S. Schacher, and V.S. Varadarajan, editors, {\em
  {Linear Algebraic Groups and their Representations}}, volume 153 of {\em
  Contemporary Mathematics}, pages 109--142, 1993.

\bibitem{GM}
Jacopo {Gandini} and Andrea {Maffei}.
\newblock {The Bruhat order on Hermitian symmetric varieties and on abelian
  nilradicals}.
\newblock {\em arXiv e-prints}, page arXiv:1708.05523, August 2017.

\bibitem{Vogan}
David~A. Vogan~Jr.
\newblock {Irreducible Characters of Semisimple Lie Groups III. Proof of
  Kazhdan-Lusztig Conjecture in the Integral Case}.
\newblock {\em Inventiones Mathematicae}, 71:381 -- 417, 1983.

\bibitem{LV}
George Lusztig and David~A. Vogan~Jr.
\newblock {Singularities of Closures of K-orbits on Flag Manifolds}.
\newblock {\em Inventiones Mathematicae}, 71:365 -- 379, 1983.

\bibitem{RS}
R.W. Richardson and T.A. Springer.
\newblock {The Bruhat order on symmetric varieties}.
\newblock {\em {Geom. Dedicata}}, (35):389--436, 1990.

\bibitem{RRS}
R.W. Richardson, Gerhard Rohrle, and Robert Steinberg.
\newblock Parabolic subgroups with abelian unipotent radical.
\newblock 1992.

\end{thebibliography}

\end{document}